\documentclass[12pt]{article}
\usepackage[utf8]{inputenc}
\usepackage{tikz}
\usepackage{amsthm}
\usepackage{amsmath}
\usepackage{amssymb}
\usepackage{amsthm}
\usepackage{amscd}
\usepackage[all]{xy}
\usepackage[margin=1in]{geometry}
\usepackage{comment}
\usepackage{pdfpages}
\usepackage{tikz}
\usepackage[T1]{fontenc}
\usepackage{amsmath,amsfonts,amssymb,enumerate}
\usepackage{amscd}
\usepackage{graphicx}
\usepackage[T1]{fontenc}
\usepackage{amsmath,amsfonts,amssymb,enumerate}
\usepackage{amsthm} 
\usepackage[english]{babel}

\newtheorem{thm}{Theorem}[section]
\newtheorem{lem}[thm]{Lemma}
\newtheorem{prop}[thm]{Proposition}
\newtheorem{cor}[thm]{Corollary}

\theoremstyle{definition}

\newtheorem{definition}[thm]{Definition}
\newtheorem{remark}[thm]{Remark}
\newtheorem{example}[thm]{Example}

\numberwithin{equation}{section}
\newcommand{\Z}{\mathbb{Z}}
\newcommand{\N}{\mathcal{N}}

\def\ker{\mbox{\rm ker}}

\def\deg{\mbox{\rm deg}}
\def\sdefect{\mbox{\rm sdefect}}

\begin{document}

\title{On quasi-equigenerated and Freiman cover ideals of graphs}


\author{Benjamin Drabkin \thanks{ University of Nebraska -- Lincoln, Lincoln Nebraska \textbf{Email address:} benjamin.drabkin@huskers.unl.edu }  \and Lorenzo Guerrieri
\thanks{Universit\`a di Catania, Dipartimento di Matematica e
Informatica, Viale A. Doria, 6, 95125 Catania, Italy \textbf{Email address:} guelor@guelan.com }} 

\maketitle

\begin{abstract}
\noindent A quasi-equigenerated monomial ideal $I$ in the polynomial ring $R= k[x_1, \ldots, x_n]$ is a Freiman ideal if $\mu(I^2) = l(I)\mu(I)- \binom{l(I)}{2}$ where $l(I)$ is the analytic spread of $I$ and $\mu(I)$ is the number of minimal generators of $I$. Freiman ideals are special since there exists an exact formula computing 
the minimal number of generators of any of their powers. In this work we address the question of characterizing which cover ideals of simple graphs are Freiman.

\medskip

\end{abstract}

\section{Introduction}
Given an homogeneous ideal $I$ in a polynomial ring $R= k[x_1, \ldots, x_n]$, it is in general a difficult problem to exactly compute the number of generators of each power of $I$. 
In the case that $I$ is generated by a regular sequence and has $\mu(I)=t$ generators, it is well-known that $ \mu(I^m)=\binom{m+t-1}{t-1} $.  This is the largest possible number of generators that $I^m$ can achieve.  At the other extreme, in some cases powers of an ideal can be generated by very few elements.  For instance, the ideals with tiny square defined in \cite{EHM} provide a family of monomial ideals whose members can be generated by arbitrarily many elements but satisfy $\mu(I^2)=9$.


However, there are large classes of monomial ideals for which the number of minimal generators of $I^n$ is well-behaved.  If all generators of $I$ have the same degree with respect to some standard (or non-standard) $\mathbb{N}$-grading of $R$, then $I$ is called \it equigenerated \rm (or \it quasi-equigenerated\rm).  Such an ideal cannot have a tiny square.  In particular, when $I$ is equigenerated or quasi-equigenerated a result of Herzog, Mohammadi Saem, and Zamani \cite[Theorem 1.9]{HMZ} shows that $\mu(I^2) \geq l(I)\mu(I)- \binom{l(I)}{2}$ where $l(I)$ is the analytic spread of $I$.
Furthermore in \cite[Proposition 1.8]{Herz}, Herzog and Zhu show a similar inequality, describing a lower bound for the number of generators of any power of $I$. These inequalities are consequence of a famous theorem of Freiman \cite{freiman}, proved in the context of additive number theory and stating that, for a finite set $X \subseteq \mathbb{Z}^n$, $$ |2X| \geq (d + 1)|X| - \binom{d+1}{2}$$ where $d$ is the dimension of the smallest affine subspace of $\mathbb{Q}^n$ containing $X$ and $2X = \{ a+b \mbox{ : } a,b \in X \} $.
The lower bound for $ \mu(I^2) $ is then obtained applying Freiman's theorem to the set of exponents of the generators of $I.$

In \cite{Herz}, Herzog and Zhu also consider the case in which this lower bound is met.  They define a \textit{Freiman ideal} to be an equigenerated monomial ideal $I$ such that $\mu(I^2) = l(I)\mu(I)- \binom{l(I)}{2}$. In \cite{HHZ}, Herzog, Hibi, and Zhu extend the same definition to quasi-equigenerated monomial ideals.  What makes a Freiman ideal very interesting is the fact that the number of generators of any power can be computed by an exact formula in terms of the number of generators of the ideal and of the analytic spread (see \cite[Corollary 1.9]{Herz}). This formula generalizes the formula existing in the case of ideals generated by regular sequences, since for them $l(I)=\mu(I)$.

The two papers \cite{Herz},\cite{HHZ} provide several characterizations of Freiman ideals in terms of their Hilbert polynomials and fiber cones. In particular \cite[Theorem 1.3]{HHZ} shows that Freiman ideals are exactly those ideals with linear Hilbert polynomials and with fiber cones having minimal multiplicity.
This is a very restrictive property which, for ideals arising from combinatorial structures, often guarantees strong combinatorial properties.  Freiman ideals in the classes of principal Borel ideals, Hibi ideals, Veronese type ideals, matroid ideals, and edge ideals of graphs have been studied in \cite{Herz} and \cite{HHZ}.   

The aim of this article is to study Freiman ideals among cover ideal of graphs. Let $G$ be a finite simple graph 
on $n$ vertices with edge set $E$.  Identifying each vertex with a variable $x_i \in R$, the cover ideal $J(G)$ is defined to be \[J(G):=\bigcap_{\{x_i,x_j\}\in E}(x_i,x_j) \subseteq R.\]
Cover ideals are squarefree monomial ideals and have been studied by many authors in the last twenty years. For an overview we refer to \cite{VT}.  For a more detailed study of the foundational results on cover ideals, we refer to \cite{HHT}.


In Section 2 we recall basic definitions and results about fiber cones, Freiman ideals and cover ideals. Furthermore, we introduce the notion of equivalent vertices in graphs and explore their relation to cover ideals.

In studying Freiman property for cover ideals, the first challenge is that cover ideals are often not quasi-equigenerated.  In Section 3 we approach this challenge. It is difficult to find a purely combinatorial characterizations of quasi-equigeneratedness but we describe several criteria and give explicit characterization for cover ideals of some classes of graphs including trees, circulant graphs, and graphs with independence number one or two. Finally, we consider the behavior of cover ideals of graphs obtained through various constructions.


In Section 4 we consider the Freiman property for cover ideals. We describe when the join of two graphs is Freiman, then we relate squarefree Freiman ideals with ideals of minors of $2 \times n$ generic matrices and we prove that, in general, graphs that are close enough to be complete have Freiman cover ideal.
Since the defining ideal of the fiber cone of Freiman ideal has a 2-linear free resolution, it is clear that, when is nonzero, it must be generated in degree 2. However, we also find graphs for which the defining ideal of the fiber cone  of the cover ideal is generated in any possible degree $n \geq 3$, and hence their cover ideals are not Freiman.

Finally, in Section 5 we characterize Freiman cover ideals among different families of graphs.  We consider the family of the pairs of complete graphs sharing a vertex, the family of circulant graphs, and the family of whiskered graphs. For all these families, we also explicitly compute the analytic spread of the cover ideals.

\section{Preliminaries}

In this section we recall definitions and several preliminary results about fiber cones, Freiman ideals and cover ideals of graphs.

\subsection{Fiber cones and Freiman ideals}

Let $I$ be an homogeneous ideal in the polynomial ring $R=k[x_1, \ldots, x_n]$, where $k$ is any field. The \it fiber cone \rm of $I$ is the standard-graded $k$-algebra $$ F(I)= \bigoplus_{i=0}^{\infty} \dfrac{I^i}{\mathfrak{m}I^i},
$$ where $\mathfrak{m}=(x_1, \ldots, x_n)$ is the homogeneous maximal ideal of $R$. The coefficients of the Hilbert series of the fiber cone are $\mbox{dim}_k(\frac{I^i}{\mathfrak{m}I^i})= \mu(I^i)$, where $\mu(I^i)$ indicates as usual the minimal number of generators of $I^i$. 
The Krull dimension of $F(I)$ is called \it analityc spread \rm and it is usually denoted by $l(I)$; this invariant measures how fast the number of generators of the powers of $I$ increase.

If the fiber cone of $I$ is isomorphic to a polynomial ring in $\mu(I)$ variables, the ideal $I$ is called of \it linear type\rm. This happens if the generators of $I$ form a regular sequence. 

In order to study the Freiman property of ideals, we will consider (quasi-)equigenerated homogeneous ideals. We give the following definition:

\begin{definition}
Let $I=(f_1, \ldots, f_t) \subseteq k[x_1, \ldots, x_n]$ be a monomial ideal. Write its monomial generators as $f_j= \prod_{i=0}^{n}x_i^{c_{i_j}}$. For $\alpha= (a_1, \ldots, a_n) \in \mathbb{N}^n$ we define $$d_{\alpha}(f_j)= \sum_{i=0}^{n} a_ic_{i_j}.$$
We say that $I$ is \it quasi-equigenerated of degree $d$ \rm if there exists $\alpha= (a_1, \ldots, a_n) \in \mathbb{N}_{\geq 1}^n$, such that, for every $j,k$, $$ d_{\alpha}(f_j)= d_{\alpha}(f_k)=d. $$ We say that $I$ is \it equigenerated \rm if all its generators have the same degree (equivalently $I$ is quasi-equigenerated taking $ \alpha= (1, 1, \ldots, 1) $).
 \end{definition}

It is known that the fiber cone of an equigenerated homogeneous ideal is an integral domain because it is isomorphic to the $k$-algebra generated by its minimal generators. We recall this result because, up to changing the grading on the polynomial ring, it can be extended to quasi-equigenerated ideals.

\begin{prop} \label{domain} \rm \cite[Proposition 4.8]{heinzer} \it
Let $R=k[x_1,\ldots,x_n]$ be a (not necessarily standard) $\mathbb{N}$-graded  polynomial ring in $n$ variables over field $k$, and let $\mathfrak{m}$ denote its homogeneous maximal ideal.  Suppose the ideal $I=(f_1,\dots,f_n)$ is homogeneous and equigenerated of degree $t$ with respect to the grading of $R$.  Then $F(I_\mathfrak{m})$ (and thus $F(I)$) is a domain.\end{prop}
\begin{proof} 
Let $I_i=I^i\cap R_{it}$ be the $it$-th graded component of $I^i$.  Then $k\oplus I_1\oplus I_2\oplus\cdots=k[f_1,\ldots,f_n]$.  Since $I$ is equigenerated with respect to the grading of $R$, $I^i/\mathfrak{m}I^i\cong I_i$ for all $i$.   Therefore \[k[f_1,\ldots,f_n] \cong \bigoplus_{i=0}^\infty \left( \dfrac{I^i}{\mathfrak{m}I^i} \right) \cong \bigoplus_{i=0}^\infty \left( \dfrac{I_\mathfrak{m}^i}{\mathfrak{m}I_\mathfrak{m}^i} \right) \cong F(I_\mathfrak{m}).\]  Thus $F(I_\mathfrak{m})$ is a domain.
\end{proof}

\medskip

We specialize now to the case of ideals that can be generated by monomials.
Let $I=(f_1, \ldots, f_t) \subseteq k[x_1, \ldots, x_n]$ be a monomial ideal. Using the homomorphism $T_i \to f_i$, 
we write the fiber cone as $$ F(I) \cong \frac{k[T_1, \ldots, T_t]}{\mathcal{I}} $$ where $k[T_1, \ldots, T_t]$ is a polynomial ring over the same field $k$.
The ideal $\mathcal{I} \subseteq k[T_1, \ldots, T_t]$ is called the \it defining ideal \rm of the fiber cone of $I.$ We describe the structure of the defining ideal in the monomial case:

\begin{prop}  \label{defidealhomog}
 Let $I=(f_1, \ldots, f_t)$ be a quasi-equigenerated monomial ideal. The defining ideal $\mathcal{I} \subseteq k[T_1, \ldots, T_t]$ of the fiber cone $F(I)$ is a homogeneous ideal, and when it is nonzero, it is generated by binomials of the form $$ T_{i_1} \cdots T_{i_r} - T_{j_1}\cdots T_{j_r}$$ and such that $ \lbrace i_1, \ldots, i_r \rbrace \cap \lbrace j_1, \ldots, j_r \rbrace = \emptyset. $
 \end{prop}
\begin{proof}
By Proposition \ref{domain}, $F(I) \cong k[f_1, \ldots, f_t]$, hence $F(I)$ is a toric ring and its defining ideal $\mathcal{I}$ is generated by binomials by \cite[Proposition 10.1.1]{HH}.
To see that $\mathcal{I}$ is homogeneous, assume by way of contradiction that it has a generator of the form $T_{i_1} \cdots T_{i_r} - T_{j_1}\cdots T_{j_s}$ with $r \neq s$. This it is equivalent to say that, if $ I=(f_1, \ldots, f_t) $, then $f_{i_1} \cdots f_{i_r} = f_{j_1}\cdots f_{j_s}$. But, since $I$ is quasi-equigenerated, we may find $\alpha \in \mathbb{N}_{\geq 1}^n$ such that $d_{\alpha}(f_i)=d_{\alpha}(f_j)$, for every $i,j$. This implies $rd_{\alpha}(f_{i_1})=sd_{\alpha}(f_{i_1})$, and hence $d_{\alpha}(f_{i_1})=0$, that is a contradiction.
\end{proof}

The next proposition states a criterion that we are going to use trough this paper, to show when the ideals of $F(I)$ generated by monomials of degree one (the variables $T_i$) are primes.

\begin{prop} \label{quotients}
Let $I=(f_1,\dots,f_t)\subseteq k[x_1,\ldots,x_n]$ be a quasi-equigenerated monomial ideal.  Let $S=k[T_1,\ldots,T_t]$, and $\mathcal{I}\subseteq S$ be the defining ideal of $F(I)$.  Let $\mathbf{T}=\{T_1,\ldots,T_t\}$ and let $\mathcal{J}$ be an ideal of $S$ generated by some subset $V\subseteq\mathbf{T}$.  Let $J$ be the image of $\mathcal{J}$ in $F(I)$. The following are equivalent:
\begin{enumerate}
\item $J$ is prime.
\item $\frac{(\mathcal{I}+\mathcal{J})}{\mathcal{J}}$ is either $(0)$ or is generated by binomials. 
\item $\frac{F(I)}{J}=\frac{S}{(\mathcal{I}+\mathcal{J})} \cong k[f_i \mbox{ \rm | } T_i\not\in V]$.
\end{enumerate}
\end{prop}

\begin{proof}
 $(1)\Rightarrow (2)$:  Let $\psi:S\rightarrow S/\mathcal{J}$ and $\rho: S\rightarrow S/\mathcal{I}$ be the natural maps.  By Proposition \ref{defidealhomog}, $\mathcal{I}$ is generated by binomials.  Let $\alpha-\beta$ be a binomial in $\mathcal{I}$ with $\alpha, \beta$ monomials.  If  $\psi(\alpha-\beta)\neq 0$, then either one or both of $\alpha,\beta$ are not in $\mathcal{J}$. 
 If this happen for both, 
 then $\psi(\alpha-\beta)$ is a binomial in $S/\mathcal{J}$.  Suppose $\alpha \not \in \mathcal{J}$, and $\beta\in \mathcal{J}$. 
 Since $\rho(\mathcal{J})=J$, we know that $\rho(\beta)\in J$ and, since $0 = \rho(\alpha-\beta)\in J$, also $\rho(\alpha)\in J$.  But $\rho(\alpha)$ is a product of elements of the form $\rho(T_i)$ where $T_i\not\in V$, and such elements are not in $J$, a contradiction.

$(2)\Rightarrow (3)$  Consider the map $\phi:S/\mathcal{J}\rightarrow k[f_i \mbox{ | } T_i\not\in V]$ given by $\overline{T_i}\mapsto f_i$.  Since $S/\mathcal{J}\cong k[\mathbf{T}\setminus V]$ we identify $S/\mathcal{J}$ with its isomorphic subring in $S$ and view each nonzero  $\overline{T_i}\in S/\mathcal{J}$ as $T_i\in S$.  We know that $\ker\phi$ consists of all polynomials in $T_i$ corresponding to relations in $\{f_i \mbox{ | } T_i\not\in V\}$.  Such elements of $S$ are contained in $(\mathcal{I}+\mathcal{J})/\mathcal{J}$, so $\ker(\phi)\subseteq\mathcal{I}$. Let $g$ be a nonzero generator of $(\mathcal{I}+\mathcal{J})/\mathcal{J}$.  Then $g$ is a binomial and $g\in\mathcal{I} \setminus \mathcal{J}.$  
Since $\mathcal{I}$ is generated by binomials corresponding to the binomial relations on $\{f_1,\dots,f_t\}$, $g$ corresponds to a binomial relation on $\{f_i \mbox{ | } T_i\not\in V\}$.  Therefore $g\in \ker(\phi)$, and thus $(\mathcal{I}+\mathcal{J})/\mathcal{J}=\ker(\phi)$.

$(3)\Rightarrow (1)$:  This is clear since $k[f_i \mbox{ | } T_i \not \in V]$ is an integral domain.
\end{proof}

\medskip

For a quasi-equigenerated monomial ideal, it has been shown in \cite[Theorem 1.1]{HHZ} that $$  \mu(I^2) \geq l(I)\mu(I) - \binom{l(I)}{2}.
$$ 
We will refer to this fact as the \it Freiman inequality. \rm
When the equality holds in the Freiman inequality, $I$ is called a \it Freiman ideal. \rm
Quasi-equigenerated monomial ideals of linear type are always Freiman ideals for which $l(I)=\mu(I)$.

Freiman ideals are characterized by the following properties. A very interesting fact obtained as a consequence, is that, if an an ideal is Freiman, the number of generators of any power can be exactly computed in term of $\mu(I)$ and $l(I)$.

\begin{thm} \rm  \cite[Theorem 1.3]{HHZ} \it
\label{herzogthm}
Let $I$ be a quasi-equigenerated monomial ideal. The following assertions are equivalent:
\begin{enumerate}
    \item $I$ is a Freiman ideal.
    \item $\mu(I^j)= \binom{l(I)+j-2}{j-1}\mu(I)-(j-1)\binom{l(I)+j-2}{j}$ for every $j \geq 1.$
    \item $\mu(I^j)= \binom{l(I)+j-2}{j-1}\mu(I)-(j-1)\binom{l(I)+j-2}{j}$ for some $j \geq 2.$
    \item The Hilbert polynomial of $F(I)$ is linear.
    \item $F(I)$ has minimal multiplicity.
    \item $F(I)$ is Cohen-Macaulay and its defining ideal $\mathcal{I}$ has a 2-linear free resolution.
\end{enumerate}
\end{thm}

\subsection{Cover ideals of graphs}
In the following we always consider simple finite graphs without multi-edges or loops.

\begin{definition}
\rm Let $G$ be a graph with vertex set $V=\{x_1,\dots,x_n\}$ and edge set $E$.  
For $m \geq 1$, we say that a monomial $g \in R$ is an $m$-cover of $G$ if for every edge $\{x_i,x_j\}\in E$, there exists one monomial of the form $h_{ij}= x_i^a x_j^b$ with $a+b \geq m$ such that $h_{ij}$ divides $g$.
We say that a vertex $m$-cover is \textit{minimal} if it is not divisible by any other different vertex $m$-cover.
\end{definition}

The generators of $J(G)$ are the monomials which correspond to the minimal vertex 1-covers of $G$. For any $m\in\mathbb{N}$, the symbolic power \[J(G)^{(m)}=\bigcap_{\{x_i,x_j\}\in E}(x_i,x_j)^m\] is generated by the monomials corresponding to the minimal vertex $m$-covers of $G$ while the ordinary power $J(G)^m$ is generated by the monomials corresponding to vertex $m$-covers which decompose into the product of $m$ vertex 1-covers.    

\begin{definition}
Given a graph $G$ with $V(G)=\{x_1,\dots,x_n\}$, an \it independent set \rm of $G$ is a subset $U \subseteq V$ such that for every $x_i, x_j \in U$, $\{ x_i, x_j\} \not \in E(G).$ We denote by $c(G)$ the \it independence number \rm of $G$, that is the maximal cardinality of an independent set of $G$. An independent set of $G$ is said \it maximal \rm if it is not contained in any other independent set. 
 \end{definition}
 
 \begin{remark} \label{indipcover}
 A set $U \subseteq V$ is an independent set of $G$ if and only if the monomial $h_U= \prod_{x_i \not \in U}x_i$ is a 1-cover of $G$ and is maximal if and only if $h_U$ is a minimal 1-cover.
 \end{remark}

 \medskip
 
 We recall the following well-known notation.
Given a graph $G=(V,E)$ and a vertex $x$, the set of \textit{neighbors} of  $x$ is the set $\N(x)$ containing all the vertices $x_j$ adjacent to $x$ (i.e. $\lbrace x_i, x_j \rbrace$ is an edge of $G$). 
We recall that the \it degree \rm of a vertex is the number of adjacent vertices and a vertex of degree 1 is called a \it leaf. \rm
A graph on $n$ vertices such that each vertex has degree $n-1$ is called \it complete. \rm
Given a set of vertices $U \subseteq V$ in $G$, the \textit{induced subgraph} on $U$ is the graph with vertex set $U$, and edge set $\{\{x,y\}\in E|x,y\in U\}$.

\begin{definition} \label{equidef}
We say that the vertices $x$ and $y$ of a graph $G$ are \it equivalent \rm if $\N(x)=\N(y).$ \\
Given two graphs $G_1=(V_1, E_1), G_2=(V_2,E_2)$ we say $G_1 \leq^{\star} G_2$ if $V_1 \subseteq V_2$, $G_1$ is the induced subgraph of $G_2$ on $V_1$, and every vertex $x \in V_2 \setminus V_1$ is equivalent to a vertex in $V_1$ as vertices of $G_2.$ 
The relation $\leq^{\star}$ defines a partial order and we 
say that a minimal elements with respect to it is a \it reduced \rm graph. A reduced graph has no equivalent vertices.
 \end{definition}

 \begin{example}
In the 4-cycle $C_4$ the pairs of opposite vertices are equivalent.

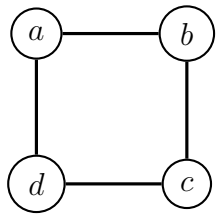
\begin{center}
\begin{tikzpicture}
\begin{scope}[every node/.style={fill=white,circle,thick,draw}]
    \node(A) at (-1,1) {$a$};
    \node(B) at  (1,1) {$b$};
    \node(C) at  (1,-1) {$c$};
    \node(D) at (-1,-1) {$d$};
\end{scope}
\begin{scope}
            [every edge/.style={draw=black,very thick}]
    \path[-](A)edge node {} (B);
    \path[-](B) edge node {} (C);
    \path[-](C) edge node {} (D);
    \path[-](D) edge node {} (A);
\end{scope}
\end{tikzpicture}
\end{center}
Vertices $a$ and $c$ are equivalent, as are vertices $b$ and $d$.
\end{example}

 Now, we show that in order to study cover ideals we can reduce to consider reduced graphs.
 
 \begin{lem} 
\label{equivalent} 
 Let $G$ be a graph, $I$ its cover ideal, and let $x,y$ be two equivalent vertices of $G$. For any minimal $m$-cover $f$ of $G$, $x^a$ divides $f$ if and only if $y^a$ divides $f$.
\end{lem}
 
\begin{proof}
Let $f$ be a minimal $m$-cover of $G$ and let $a \geq 0$ be the largest power such that $x^a$ divides $f$. Hence there exists $z \in \N(x)=\N(y)$ such that $z^{m-a}$ divides $f$ and $z^{m-a+1}$ does not divide $f$. If follows that $y^a$ divides $f$.
\end{proof} 

\begin{thm}
\label{sdefequal}
Let $G_1 \leq^{\star} G_2$ be two graphs and let $I_1$ and $I_2$ be their cover ideals. Then their fiber cones $F(I_1)$ and $F(I_2)$ are isomorphic. 
\end{thm}

\begin{proof}
Let $V_1= \lbrace x_1, \ldots, x_n \rbrace$ be the set of vertices of $G_1.$ We may assume the set of vertices of $G_2$ to be $V_2= \lbrace x_1, \ldots, x_n, y \rbrace$ (if there are more vertices it is possible to iterate the same argument of this proof).
Let $I_2 \subseteq k[x_1, \ldots, x_n, y]$ be the cover ideal of $G_2$. Since $G_1 \leq^{\star} G_2$, there exists $x_i \in V_1$ equivalent to $y$. The thesis follows by Lemma \ref{equivalent} via the change of variables $T:=x_iy$.
\end{proof}

\begin{remark} \label{samesdefect}
 The ideals $I_1$ and $I_2$ of the preceding theorem share many properties. For instance, for every $m \geq 1$, $\mu(I_1^m)=\mu(I_2^m)$, $\mu(I_1^{(m)})=\mu(I_2^{(m)})$ and sdef$(I_1,m)=$sdef$(I_2,m)$. The quantity sdef$(I,m)$ is called \it symbolic defect \rm and it is defined as the minimal number of generators of the module $\frac{I^{(m)}}{I^m}$. The paper \cite{DG} is devoted to the study of symbolic defect of cover ideals of graphs.
 \end{remark}

\section{Quasi-equigenerated cover ideals}

In this section, we approach the question of understanding which cover ideals of graphs are quasi-equigenerated.
 
Easy computations allow to observe that the cover ideal of any graph with 3 or 4 vertices is quasi-equigenerated while
the unique graph with 5 vertices and non-quasi-equigenerated cover ideal is the path $P_5$ (see below). 
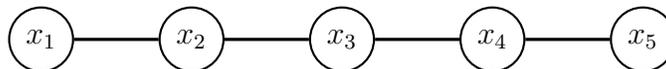
\begin{center}
\begin{tikzpicture}
\begin{scope}[every node/.style={fill=white,circle,thick,draw}]
    \node(A) at (-4,0) {$x_1$};
    \node(B) at  (-2,0) {$x_2$};
    \node(C) at  (0,0) {$x_3$};
    \node(D) at (2,0) {$x_4$};
    \node (E) at (4,0) {$x_5$};
\end{scope}
\begin{scope}
            [every edge/.style={draw=black,very thick}]
    \path[-](A)edge node {} (B);
    \path[-](B) edge node {} (C);
    \path[-](D) edge node {} (C);
    \path[-](D) edge node {} (E);
\end{scope}
\end{tikzpicture}
\end{center}
Indeed the cover ideal of $P_5$ is $$J(P_5)= (x_1x_3x_5, x_1x_3x_4,  x_2x_4, x_2x_3x_5)$$ and, assuming the existence of $\alpha \in \mathbb{N}_{\geq 1}^n$ such that,  $ d_{\alpha}(f_j)= d_{\alpha}(f_k)$, we get the relations $a_4 = a_5$, $a_4 = a_3 + a_5$ deriving the contradiction $a_3=0.$

\begin{example}
Let $G$ be a graph whose cover ideal $I$ is generated in only two different degrees $m_1 < m_2.$ Assume that there exists $x_i$ dividing all the minimal 1-covers of degree $m_1$ but not dividing any of the minimal 1-covers of degree $m_2$. It follows that $I$ is quasi-equigenerated setting $a_i=m_2-m_1+1$ and $a_k=1$ for $k \neq i$.
 \end{example}

In general it seems not easy to find an exact classification of all the graphs having quasi-equigenerated cover ideal. We provide here several criteria and the complete characterization for some families of graphs. First we observe that for this study we can only consider graphs reduced in the sense of Definition \ref{equidef}. For a graph $G=(V,E)$ and a vertex $x \in V$, we denote by $G\setminus \{x\}$ the induced subgraph on $V\setminus \{x\}$.

\begin{lem}
\label{reduced}
Let $G$ be a $n$-vertex graph with equivalent vertices $x_n$ and $x_{n-1}$.  Then $J(G)$ is quasi-equigenerated if and only if $J(G\setminus \{x_n\})$ is quasi-equigenerated. 
\end{lem}
\begin{proof}
Suppose that $J(G\setminus \{x_n\})=(f_1,\dots,f_s)$ is quasi-equigenerated.  Then there exists some $\alpha=(a_1,\dots,a_{n-1})\in\mathbb{N}_{\geq 1}^{n-1}$ such that $d_\alpha(f_i)=d_\alpha(f_j)$ for all $1\leq i,j\leq s$.  Let $\beta \in \mathbb{N}_{\geq 1}^n$ be defined by $(2a_1,\dots,2a_{n-2},a_{n-1},a_{n-1})$.  We know by Lemma \ref{equivalent} that each generator of $J(G)$ is of the form $x_n^{t_i}f_i$ for some $1\leq i\leq s$ where $t_i$ is the highest power of $x_{n-1}$ dividing $f_i$ ($t_i \in \lbrace 0,1 \rbrace$).  Thus $d_{\beta}(x_n^{t_i}f_i)=a_{n-1}t_i + 2d_{\alpha}(f_i)-a_{n-1}t_i = 2d_{\alpha}(f_i)$ and therefore $J(G)$ is quasi-equigenerated. 

Suppose $J(G)=(g_1,\dots,g_s)$ is quasi-equigenerated. We know by Lemma \ref{equivalent} that $J(G\setminus\{x_n\})=(\frac{g_1}{x_n^{t_1}},\dots,\frac{g_s}{x_n^{t_s}})$ where for each $i$, $t_i$ is the highest power of $x_n$ dividing $g_i$.  Since $J(G)$ is quasi-equigenerated, there exists some $\alpha=(a_1,\dots,a_n) \in \mathbb{N}_{\geq 1}^n$ such that $d_\alpha(g_i)=d_\alpha(g_j)$ for all $1\leq i,j\leq s$.  Let $\beta=(a_1,\dots,a_{n-1}+a_n)$.  Then $d_\beta 
(\frac{g_i}{x_n^{t_i}}) =d_\alpha(g_i)$ for all $i$, and hence $J(G\setminus\{x_n\})$ is quasi-equigenerated.
\end{proof}

The cover ideals of graphs of independence number two are always quasi-equigenerated. 

\begin{prop}
\label{quasic2} Let $G$ be a graph such that $c(G)=2$ and let $I=J(G)$ be its cover ideal. Then $I$ is quasi-equigenerated. 
\end{prop}
 
 \begin{proof}
 After relabeling the vertices, let $ x_1, \ldots, x_c $ be the vertices of $G$ of maximal degree and let $F=x_1x_2 \cdots x_n$ be the product of all the variables. The minimal 1-covers of $G$ are of the form $Fx_i^{-1}$ for $i \leq c$ and $F(x_lx_j)^{-1}$ where $\lbrace x_l, x_j \rbrace$ is an independent set of $G$ of cardinality two and $j,l > c$. In the case there are no vertices of $G$ having maximal degree, then each vertex is contained in an independent set of cardinality two, and therefore $I$ is equigenerated. Otherwise, set $a_i=2$ for $i \leq c$ and $a_i=1$ for $i > c$. Hence, setting $\alpha=(a_1, a_2, \ldots, a_n)$, we get for $i \leq c$, $$ d_{\alpha}(Fx_i^{-1})= 2(c-1)+(n-c) $$ and for $j,l > c$, $$ d_{\alpha}(F(x_lx_j)^{-1})= 2c + (n-c-2). $$ This implies $I$ quasi-equigenerated. 
\end{proof}

Next result characterizes graphs with equigenerated cover ideals in term of independent sets. 

\begin{prop} 
\label{equigen} 
Let $G=(V,E)$ be a graph of $n$ vertices and let $I$ be its cover ideal. The following conditions are equivalent:
\begin{enumerate}
\item $I$ is equigenerated.
\item All the maximal independent sets of $G$ have the same cardinality.
\end{enumerate}
\end{prop}

\begin{proof}
It is a straightforward consequence of the definitions and of Remark \ref{indipcover}. 
\end{proof} 




Next lemma describes a useful way to detect non-quasi-equigenerated cover ideals considering the cover ideals of particular induced subgraphs.

\begin{definition} \label{gi}
Let $G=(V,E)$ be a graph of $n$ vertices and let $x_i \in V$. We call $G_i$ the induced subgraph on the set $V \setminus (\N(x_i) \cup \lbrace x_i\rbrace)$.
 \end{definition}

\begin{lem} 
\label{quasi1} 
Let $G=(V,E)$ be a graph of $n$ vertices and let $I=(f_1, \ldots, f_t)$ be its cover ideal. 
Then if $I$ is quasi-equigenerated (resp. equigenerated), the cover ideal $J(G_i)$ is quasi-equigenerated (resp. equigenerated) for every $i$.
\end{lem}
 
\begin{proof}
Let $h_j$ be a minimal 1-cover $h_j$ of $G_i$ and let $h$ be the product of the neighbors of $x_i$ in $G$. Hence $f_j=hh_j$ is a minimal 1-cover of $G$ and the minimal 1-covers of $G$ not divisible by $x_i$ are all of this form. Let $V \setminus (\N(x_i) \cup \lbrace x_i\rbrace)= \lbrace x_{i_1}, \ldots, x_{i_c} \rbrace.$ It follows that, if $I$ is quasi-equigenerated with $\alpha= (a_1, \ldots, a_n)$, $J(G_i)$ is quasi-equigenerated with $(a_{i_1}, \ldots, a_{i_c})$. 
\end{proof} 

The converse of this result is not true, since one can see that the cover ideal of the 6-cycle $J(C_6)$ is not quasi equigenerated, but $J((C_6)_i)$ is quasi-equigenerated for every of its vertices $x_i$. In Theorem \ref{quasicirculant}, we explicitly characterize which circulant graphs have quasi-equigenerated cover ideals.

\medskip


\begin{remark} \label{criterion}
Lemma \ref{quasi1} gives rise to a criterion for non-quasi-equigeneratedness.  Let $G$ be a graph containing an induced subgraph $A$ such that $J(A)$ is non-quasi-equigenerated.
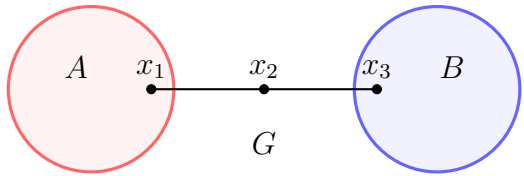
\begin{center}
\begin{tikzpicture}
\filldraw[color=red!60, fill=red!5, very thick] (-.8,0) circle (1.1);
\filldraw[color=blue!60, fill=blue!5, very thick] (3.8,0) circle (1.1);
\node[above, black] at (-1,0) {$A$};
\node[above, black] at (4,0) {$B$};
\node[above, black] at (1.5,-1) {$G$};
\draw[fill] (0,0) circle [radius=0.06];
\node[above,black] at (0,0) {$x_1$};
\draw[thick] (0,0) -- (3,0);
\draw[fill] (1.5,0) circle [radius=.06];
\node[above, black] at (1.5,0) {$x_2$};
\draw[fill] (3,0) circle [radius=.06];
\node[above,black] at (3,0) {$x_3$};
\end{tikzpicture}
\end{center}
\
If there exist vertices $x_1,x_2,x_3$ with $x_1\in V(A)$, such that the induced subgraph on $\{x_1,x_2,x_3\}$ is $P_3$, then the graph $G_3$ (obtained following Definition \ref{gi}) will have $A$ as a connected component with non-quasi-equigenerated cover ideal.
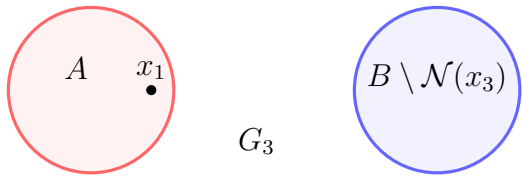
\begin{center}
\begin{tikzpicture}
\filldraw[color=red!60, fill=red!5, very thick] (-.8,0) circle (1.1);
\filldraw[color=blue!60, fill=blue!5, very thick] (3.8,0) circle (1.1);
\node[above, black] at (-1,0) {$A$};
\node[above, black] at (3.8,-.2) {$B\setminus\mathcal{N}(x_3)$};
\node[above, black] at (1.4,-1) {$G_3$};
\draw[fill] (0,0) circle [radius=0.06];
\node[above,black] at (0,0) {$x_1$};
\end{tikzpicture}
\end{center}
Thus by Lemma \ref{quasi1}, $J(G)$ is not quasi-equigenerated.
\end{remark}

We describe which trees have quasi-equigenerated cover ideal. We recall that a \it tree \rm is a graph not containing any induced cyclic subgraph.

\begin{thm}
\label{Trees}
Let $T$ be a tree.  $J(T)$ is quasi-equigenerated if and only if every vertex of degree at least $2$ is adjacent to a leaf.
\end{thm}
\begin{proof}
 Suppose that $T$ contains a vertex $v$ of degree at least $2$ which is not adjacent to any leaves.  Then there exist vertices $a,b,x,y$ in $T$ such that $(a,b)$, $(b,v)$, $(v,x)$, and $(x,y)$ are edges in $T$.  Let $H$ be the set of all leaves adjacent to $\{a,b,x,y\}$, and consider the induced subgraph $S$ of $T$ on $\{a,b,v, x,y\}\cup H$. Let $x_{1}, \ldots, x_{s}$ be the vertices of $T$ having exactly distance two from at least one vertex among $\{a,b,v, x,y\}$. Following the notation of Definition \ref{gi}, set $G^1:=T_{1}$ and for $j=2, \ldots s$, set $G^j:=G_{j}^{j-1}$ (clearly $x_{1}, \ldots, x_{s}$ are pairwise not adjacent since $T$ is a tree and hence $x_{j}$ is a vertex of $G^{j-1}$). The last graph obtained with this process is $G^s=S$. By iterated applications of Remark \ref{criterion} to the the graphs $G^{j}$ we get that, if $J(S)$ is not quasi-equigenerated then also $J(T)$ is not quasi-equigenerated.  
 By Lemma \ref{reduced} we may assume that $S$ has no two equivalent vertices.  The assumption of $v$ not adjacent to any leaf implies that $S$ is one of the five following graphs: 
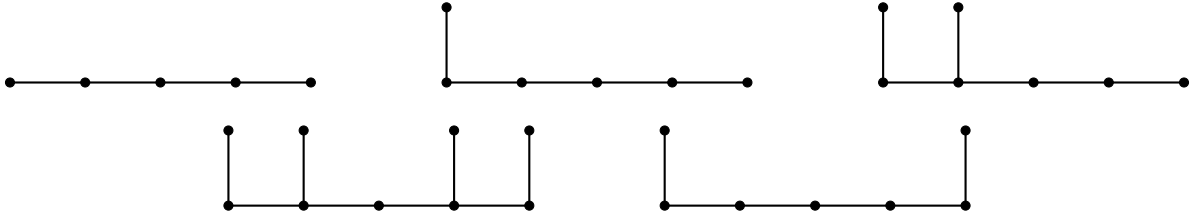
\begin{center}
\begin{tabular}{ccccc}
    \begin{tikzpicture}
\draw[fill](-2,0) circle [radius=.06];
\draw[fill] (-1,0) circle [radius=.06];
\draw[fill] (0,0) circle [radius=.06];
\draw[fill] (1,0) circle [radius=.06];
\draw [fill] (2,0) circle [radius=.06];
\draw[thick] (-2,0)--(2,0);
\end{tikzpicture} & \qquad \qquad & \begin{tikzpicture}
\draw[fill](-2,0) circle [radius=.06];
\draw[fill] (-1,0) circle [radius=.06];
\draw[fill] (0,0) circle [radius=.06];
\draw[fill] (1,0) circle [radius=.06];
\draw [fill] (2,0) circle [radius=.06];
\draw[thick] (-2,0)--(2,0);
\draw[fill](-2,1) circle[radius=.06];
\draw[thick] (-2,0)--(-2,1);
\end{tikzpicture} & \qquad \qquad &\begin{tikzpicture}
\draw[fill](-2,0) circle [radius=.06];
\draw[fill] (-1,0) circle [radius=.06];
\draw[fill] (0,0) circle [radius=.06];
\draw[fill] (1,0) circle [radius=.06];
\draw [fill] (2,0) circle [radius=.06];
\draw[thick] (-2,0)--(2,0);
\draw[fill](-1,1) circle[radius=.06];
\draw[thick] (-1,0)--(-1,1);
\draw[fill](-2,1) circle[radius=.06];
\draw[thick] (-2,0)--(-2,1);
\end{tikzpicture}\\
\end{tabular}
\end{center}
\begin{center}
\begin{tabular}{ccc}
    \begin{tikzpicture}
\draw[fill](-2,0) circle [radius=.06];
\draw[fill] (-1,0) circle [radius=.06];
\draw[fill] (0,0) circle [radius=.06];
\draw[fill] (1,0) circle [radius=.06];
\draw [fill] (2,0) circle [radius=.06];
\draw[thick] (-2,0)--(2,0);
\draw[fill](-1,1) circle[radius=.06];
\draw[thick] (-1,0)--(-1,1);
\draw[fill](-2,1) circle[radius=.06];
\draw[thick] (-2,0)--(-2,1);
\draw[fill](1,1) circle[radius=.06];
\draw[thick] (1,0)--(1,1);
\draw[fill](2,1) circle[radius=.06];
\draw[thick] (2,0)--(2,1);
\end{tikzpicture} & \qquad\qquad&\begin{tikzpicture}
\draw[fill](-2,0) circle [radius=.06];
\draw[fill] (-1,0) circle [radius=.06];
\draw[fill] (0,0) circle [radius=.06];
\draw[fill] (1,0) circle [radius=.06];
\draw [fill] (2,0) circle [radius=.06];
\draw[thick] (-2,0)--(2,0);
\draw[fill](2,1) circle[radius=.06];
\draw[thick] (2,0)--(2,1);
\draw[fill](-2,1) circle[radius=.06];
\draw[thick] (-2,0)--(-2,1);
\end{tikzpicture}
\end{tabular}
\end{center}

By inspection, with the help of Lemma \ref{quasi1}, we get that the cover ideals of these five graphs are not quasi-equigenerated and hence also $J(T)$ is not quasi-equigenerated.

Suppose every degree 2 vertex of $T$ is adjacent to a leaf.  By Lemma \ref{reduced}, we may assume $T$ to be reduced in the sense of Definition \ref{equidef}. 
Thus every vertex of $T$ is either a leaf or adjacent to exactly one leaf.  Therefore $T$ is a whiskered  graph. We prove that cover ideals of whiskered graphs are equigenerated in Proposition \ref{genwhisk}.
\end{proof}

\subsection{Quasi-equigenerated circulant graphs}

 It is easy to observe that complete graphs have equigenerated cover ideal and that the cyclic graph $C_n$ has equigenerated cover ideal only for $n=3,4,5,7$.
 A natural generalization of complete graphs and cyclic graphs is given by the class of circulant graphs. Results about ideals related to circulant graphs are given for instance in \cite{Rin}, \cite{VVW}, \cite{EVV}. Here we are interested in characterizing which circulant graph has quasi-equigenerated cover ideal.

\begin{definition} 
\label{circ}
\rm Let $n$ be a positive integer and let $1 \leq s \leq \lfloor{n/2} \rfloor$. Denote by $\Z_n$ the cyclic group with $n$ elements let $S= \lbrace 1,2, \ldots s \rbrace \subseteq \Z_n.$

The circulant graph $C_n(1,2, \ldots, s)$ is defined as the graph with vertex set $\lbrace x_1, \ldots, x_n \rbrace$ and with edge set formed by the edges $\lbrace x_i, x_{i+j} \rbrace$ such that $j \in \lbrace \pm 1, \pm 2, \ldots, \pm s \rbrace$ with the sums taken modulo $n$. 
Note that the cycle $C_n$ is equal to $C_n(1)$ and the complete graph $K_n$ is $C_n(1,2, \ldots, \lfloor{n/2} \rfloor)$.  \end{definition}

The graph in the next picture is the circulant graph $C_6(1,2)$.
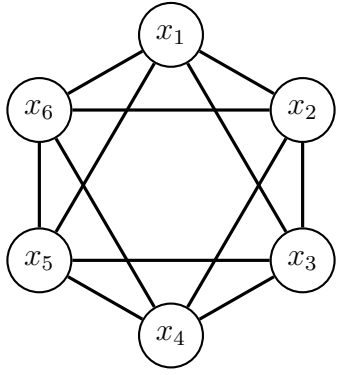
\begin{center}
\begin{tikzpicture}
\begin{scope}[every node/.style={fill=white,circle,thick,draw}]
    \node(A) at (0,2) {$x_1$};
    \node(B) at  (1.75,1) {$x_2$};
    \node(C) at  (1.75,-1) {$x_3$};
    \node(D) at (0,-2) {$x_4$};
    \node (E) at (-1.75,1) {$x_6$};
    \node (F) at (-1.75,-1) {$x_5$};
\end{scope}
\begin{scope}
            [every edge/.style={draw=black,very thick}]
    \path[-](A)edge node {} (B);
    \path[-](B) edge node {} (C);
   \path[-](A) edge node {} (E);
    \path[-](D) edge node {} (F);
    \path[-](D) edge node {} (C);
    \path[-](F) edge node {} (E);
    \path[-](A) edge node {} (C);
    \path[-](A) edge node {} (F);
    \path[-](B) edge node {} (D);
    \path[-](B) edge node {} (E);
    \path[-](C) edge node {} (F);
    \path[-](D) edge node {} (E);
\end{scope}
\end{tikzpicture}
\end{center}
Our strategy is to apply Lemma \ref{quasi1} and consider the induced subgraphs of circulant graphs of the form $G_i$. For this purpose, we need to study the cover ideals of the family of graphs we now introduce:
 
 \begin{definition}
Given two positive integers $n,s \geq 1$, we define the graph $P_{(n,s)}= (V,E)$ where $V=\lbrace x_1, \ldots, x_n \rbrace$ and $$E:= \lbrace (x_i, x_j) \, | \, |i-j| \leq s \rbrace.$$
Notice that for $s=1$, $P_{(n,1)}= P_n$ is the path on $n$ vertices.
 \end{definition}
 
 \begin{prop} 
\label{equipath} 
The cover ideal of the graph $P_{(n,s)}$ is equigenerated if and only if either $n \leq s+1$ or $n=2s+2.$
\end{prop}
 
\begin{proof}
When $n \leq s+1$, the graph is complete and hence its cover ideal is equigenerated of degree $n-1$. If $n=2s+2$ we argue in the following way: observing that $ \N(x_{s+1})= V \setminus \lbrace x_{2s+2} \rbrace $, there exists a unique minimal 1-cover of $P_{(2s+2,s)}$ not divisible by $x_{s+1}$ an it has degree $2s$. For the same reason, since $ \N(x_{s+2})= V \setminus \lbrace x_{1} \rbrace $, also the unique minimal 1-cover not divisible by $x_{s+2}$ has degree $2s$. All the remaining minimal 1-covers are of the form $x_{s+1}x_{s+2}h_1h_2$ where $h_1$ is a minimal 1-cover of the induced subgraph on $ \lbrace x_1, \ldots, x_{s} \rbrace $ and $h_2$ is a minimal 1-cover of the induced subgraph on $ \lbrace x_{s+3}, \ldots, x_{2s+2} \rbrace $. But both these induced subgraphs are complete graphs with $s$ vertices and hence $$ \deg (x_{s+1}x_{s+2}h_1h_2) = 2+ \deg(h_1)+ \deg(h_2) = 2+ 2(s-1)= 2s. $$ 
To show that all the other graphs of the family are not equigenerated first consider the case where $s+1 \leq n \leq 2s+1$.  
 If $n$ is even we have for $1 \leq j \leq n,$ $$ \left| \frac{n}{2}-j \right| \leq \left| \frac{2s+1}{2}-j \right| = \left| s+ \frac{1}{2}-j \right| \leq s $$ and thus $(x_{\frac{n}{2}}, x_j) \in E$ for every $j$. Using the symmetry of the graph, we get also $(x_{\frac{n}{2}+1}, x_j) \in E$ for every $j$. Now, the unique minimal 1-cover not divisible by $x_{\frac{n}{2}}$ has degree $n-1$ but, since the induced subgraphs on $ \lbrace x_1, \ldots, x_{\frac{n}{2}-1} \rbrace $ and on $ \lbrace x_{\frac{n}{2}+2}, \ldots, x_{n} \rbrace $ are complete (and have the same number of vertices), any minimal 1-cover divisible by $x_{\frac{n}{2}}x_{\frac{n}{2}+1}$ has degree $2+ 2(\frac{n}{2}-2)= n-2$ and therefore the cover ideal is not equigenerated. Similarly, if $n$ is odd we get $(x_{\frac{n+1}{2}}, x_j) \in E$ for every $j$ and we show that the cover ideal is not equigenerated in an analogous way. \\
 Finally we observe that if $G=P_{(n,s)}$, then, for $0 \leq j \leq s$, $$G_{n-j} = P_{(n-s-1-j,s)}.$$ It follows that by Lemma \ref{quasi1}, since $J(P_{(s+2,s)})$ is not equigenerated, then $J(P_{(s+2+(s+1+j),s)})$ is not equigenerated and therefore $J(P_{(2s+3,s)}),J(P_{(2s+4,s)}), \ldots, J(P_{(3s+2,s)})$ are not equigenerated. Since $2s+3 + (s+1+j) \geq 3s+4 $, we can iterate this last process and conclude that for $n \geq 2s+3$, $J(P_{(n,s)})$ is not equigenerated.
\end{proof}

By convention we assume the cover ideals of a graph without edges to be equigenerated.
 
\begin{thm}
\label{quasicirculant}
Let $I$ be the cover ideal of a circulant graph $G=C_n(1,2, \ldots, s)$. 
 The following conditions are equivalent:
\begin{enumerate}
\item $I$ is equigenerated.
\item $I$ is quasi-equigenerated.
\item $J(G_1)$ is equigenerated.
\item Either $ s \geq \frac{n-1}{3} $ or $s=\frac{n-3}{4}.$
\end{enumerate}
\end{thm}

\begin{proof}
$(1) \Rightarrow (2)$ is trivial. \\
$(1) \Rightarrow (3)$ follows by Lemma \ref{quasi1}. \\
$(2) \Rightarrow (1)$  By Lemma \ref{quasi1}, since $I$ is quasi-equigenerated with $\alpha= (a_1, \ldots, a_n)$, $J(G_1)$ is quasi-equigenerated with $(a_{i_1}, \ldots, a_{i_c})$ correspondent to the vertices $ \lbrace x_{i_1}, \ldots, x_{i_c} \rbrace $ of $G_1$. But, by the symmetry of the circulant graphs, after relabeling the vertices, $G_1 = G_i$ for every $i$. Thus, we can set the same values $a_{i_l}$ on the vertices of $G_2$ preserving the order given to those of $G_1$. Now observe that the vertices of $G_2$ are $$\lbrace x_{i_1+1}, \ldots, x_{i_c+1} \rbrace$$ where the sums $i_l + k$ are taken modulo $n$.
Hence $ a_{i_l} = a_{i_l+1} $ for every $1 \leq l \leq c$. Using inductively this argument, we find that $a_i = a_l$ for every $i,l$ and therefore $I$ is equigenerated.\\
$(3)\Rightarrow(1)$ Any minimal 1-cover of $G$ not divisible by the variable $x_1$ is of the form $hh_j$ where $h$ is the product of the neighbors of $x_1$ and $h_j$ is a minimal 1-cover of $G_1$. Since $J(G_1)$ is equigenerated, we get these last covers are all equigenerated. Since all the vertices $x_i$, have the same number of neighbors and $G_1 = G_i$ for every $i$, we get that $I$ is equigenerated. \\
$(3) \Leftrightarrow (4)$ Observe that $G_1 = P_{(n-2s-1,s)}$ and conclude applying Proposition \ref{equipath}.
\end{proof}

\subsection{Quasi-equigenerated join of graphs}

Our next aim is to characterize how quasi-equigeneratedness of the cover ideal behaves with respect to graph operations. We consider the operation of adding edges between vertices of two graphs.

\begin{definition} \label{joint}
Let $G_1=(V_1,E_1)$ and $G_2=(V_2, E_2)$ be two graphs where $V_1= \lbrace x_1, \ldots, x_n \rbrace$ and $V_2= \lbrace y_1, \ldots, y_m \rbrace$ are two disjoint sets of vertices. Given two non empty subsets $U_1 \subseteq V_1$ and $U_2 \subseteq V_2$, we define the graph $$ G_1 \oplus_{U_1,U_2} G_2 := (V_1 \cup V_2, E_1 \cup E_2 \cup D) $$ where $$ D:= \lbrace (x_i, y_j) \mbox{ | } x_i \in U_1, y_j \in U_2 \rbrace. $$ When $U_1=V_1$ and $U_2=V_2$ we simply denote $G_1 \oplus G_2:= G_1 \oplus_{V_1,V_2} G_2.$ This last graph is sometimes called the \it join \rm of $G_1$ and $G_2$. 
 \end{definition}

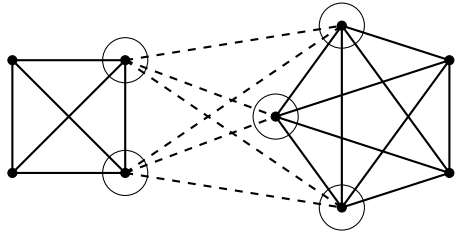
\begin{center}
    \begin{tikzpicture}
        \draw[fill](-2.5,0) circle [radius=.06];
\draw[fill] (-1,0) circle [radius=.06];
\draw[fill] (-1,1.5) circle [radius=.06];
\draw[fill] (-2.5,1.5) circle [radius=.06];

\draw[thick] (-2.5,0)--(-1,0);
\draw[thick] (-2.5,0)--(-1,1.5);
\draw[thick] (-2.5,1.5)--(-1,1.5);
\draw[thick] (-1,0)--(-1,1.5);
\draw[thick] (-1,0)--(-2.5,1.5);
\draw[thick] (-2.5,0)--(-2.5,1.5);

\draw[fill] (1,.75) circle [radius=.06];
\draw[fill] (1.88,1.96) circle [radius=.06];
\draw[fill] (1.88,-.46) circle [radius=.06];
\draw[fill] (3.31,1.5) circle [radius=.06];
\draw[fill] (3.31,0) circle [radius=.06];

\draw[thick](1,.75) --(1.88,1.96);
\draw[thick](1,.75) --(1.88,-.46);
\draw[thick](1,.75) --(3.31,1.5);
\draw[thick](1,.75) --(3.31,0);
\draw[thick](1.88,1.96) --(1.88,-.46);
\draw[thick](1.88,1.96) --(3.31,1.5);
\draw[thick](1.88,1.96) --(3.31,0);
\draw[thick](1.88,-.46)--(3.31,1.5);
\draw[thick](1.88,-.46) --(3.31,0);
\draw[thick](3.31,1.5) --(3.31,0);

\draw[thick,dashed](-1,0)--(1,.75);
\draw[thick,dashed](-1,0)--(1.88,-.46);
\draw[thick,dashed](-1,0)--(1.88,1.96);
\draw[thick,dashed](-1,1.5)--(1,.75);
\draw[thick,dashed](-1,1.5)--(1.88,-.46);
\draw[thick,dashed](-1,1.5)--(1.88,1.96);

\draw (-1,0) circle [radius=.3];    
\draw (-1,1.5) circle [radius=.3];  
\draw (1,.75) circle [radius=.3];    
\draw (1.88,-.46) circle [radius=.3];    
\draw (1.88,1.96) circle [radius=.3];
    \end{tikzpicture}
\end{center}

\begin{definition} \label{linkdef}
We say that $G_1 \oplus_{U_1,U_2} G_2$ is \it linked with covers \rm if $ p_1 = \prod_{x_i \in U_1} x_i$ and $ p_2 = \prod_{y_j \in U_2} y_j$ are respectively 1-covers of $G_1$ and $G_2$ (not necessarily minimal).
 \end{definition}
 
  \begin{prop} 
\label{linked} 
Let $G_1 \oplus_{U_1,U_2} G_2$ be linked with covers and let $I_1=(f_1, \ldots, f_s) \subseteq k[x_1, \ldots, x_n]$ and $I_2=(g_1, \ldots, g_t) \subseteq k[y_1, \ldots, y_m]$ be respectively the cover ideals of $G_1$ and $G_2$. The cover ideal $I$ of $G$ is contained in $ k[x_1, \ldots, x_n, y_1, \ldots, y_m] $ and it is generated by: $$ I = (f_1p_2, \ldots, f_sp_2, g_1p_1, \ldots, g_tp_1).$$ The monomial $p_1p_2$ is a minimal 1-cover if and only if $p_1$ and $p_2$ are both minimal 1-covers of $G_1$ and $G_2$. The others are all minimal.
\end{prop}
 
\begin{proof}
If $p_1$ and $p_2$ are both minimal 1-covers, $p_1p_2$ is clearly a minimal 1-cover. If one of them, say $p_1$ is not minimal and it is divisible by a 1-cover $h$ of $G_1$, then $hp_2$ is a 1-cover of $G$ dividing $p_1p_2$. For $f_i \neq p_1$, the monomial $ f_ip_2 $ is clearly a minimal 1-cover since there must exist $x_l$ dividing $p_1$ but not $f_i$, and therefore we cannot remove any variable $y_j$ from $p_2$, since otherwise we would uncover the edge $(x_l, y_j)$. Similarly we can see that all the monomials of the form $g_jp_1$ are minimal 1-covers if $g_j \neq p_2.$ It is easy to observe that any cover not of this form is not minimal.
\end{proof}

\begin{thm}
\label{quasilinked}
Let $G_1 \oplus_{U_1,U_2} G_2$ be linked with covers and take $p_1, p_2$ as in Definition \ref{linkdef}. Call $I$ the cover ideal of $G$ and $I_i$ the cover ideal of $G_i$ for $i=1,2$. Assume $p_1,p_2$ are either both minimal 1-covers of the respective graphs or both non-minimal. The following conditions are equivalent:
\begin{enumerate}
\item $I$ is quasi-equigenerated.
\item Both $I_1$ and $I_2$ are quasi-equigenerated.
\end{enumerate}
\end{thm}

\begin{proof}
$(1)\Rightarrow(2)$ We argue by way of contradiction and assume $I_1=(f_1, \ldots, f_s)$ not quasi-equigenerated. Assuming $p_1,p_2$ both minimal or both non-minimal, we get by Proposition \ref{linked} that for every $i$, $f_ip_2$ is a minimal 1-cover of $G$. Since $I_1$ is not quasi-equigenerated the linear system defined by the equations $ d_{\alpha}(f_i)= d_{\alpha}(f_j) $ has no solutions among the nonzero positive integers. It follows that also the linear system defined by equations $ d_{\alpha}(f_ip_2)= d_{\alpha}(f_jp_2) $ has no solutions, and hence $I$ is not quasi-equigenerated. \\
$(2)\Rightarrow(1)$ Let $I_1=(f_1, \ldots, f_s)$ be quasi-equigenerated with $\alpha=(a_1, \ldots, a_n)$ and $I_2=(g_1, \ldots, g_t)$ with $\beta=(b_1, \ldots, b_m)$. If $p_1,p_2$ are both minimal 1-covers, by Proposition \ref{linked}, it is easy to observe that $I$ is quasi-equigenerated with $(a_1, \ldots, a_n, b_1, \ldots, b_m)$, since $ d_{\alpha}(f_i)= d_{\alpha}(p_1) $ and $ d_{\beta}(g_j)= d_{\beta}(p_2) $ for every $i,j$. Instead, if $p_1,p_2$ are both non-minimal, the numbers $A=  d_{\alpha}(p_1) - d_{\alpha}(f_i) $ and $B=  d_{\beta}(p_2) - d_{\beta}(g_j) $ are both greater than zero. Hence, the ideal $I$ is quasi-equigenerated with $(ca_1, ca_2, \ldots, ca_n, db_1, db_2, \ldots, db_m)$ where $c= \frac{B}{\gcd(A,B)}$ and $d= \frac{A}{\gcd(A,B)}$.
\end{proof}

\begin{remark} \label{linkremark}
Observe that the argument used to prove Theorem \ref{quasilinked} actually shows that if $J(G_1)$ and $J(G_2)$ are not quasi-equigenerated then $J(G)$ cannot be quasi-equigenerated. Anyway, if one of them has quasi-equigenerated cover ideal, the assumptions on minimality of $p_1$ and $p_2$ are needed, since it is possible to produce a counterexample in the case one is minimal as a 1-cover and the other is not. The counterexample is the graph $G=P_5 \oplus_{U_1,V_2} P_2$ where $U_1=\lbrace x_2, x_3, x_5 \rbrace \subseteq \lbrace x_1, x_2, x_3, x_4, x_5 \rbrace$ and $V_2= \lbrace y_1, y_2\rbrace.$  
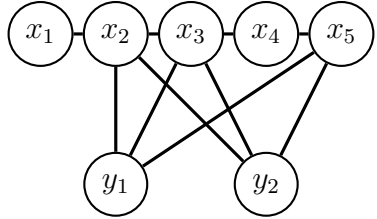
\begin{center}
\begin{tikzpicture}
\begin{scope}[every node/.style={fill=white,circle,thick,draw}]
    \node(A) at (-2,0) {$x_1$};
    \node(B) at  (-1,0) {$x_2$};
    \node(C) at  (0,0) {$x_3$};
    \node(D) at (1,0) {$x_4$};
    \node (E) at (2,0) {$x_5$};
    \node (F) at (-1,-2) {$y_1$};
    \node (G) at (1,-2) {$y_2$};
\end{scope}
\begin{scope}
            [every edge/.style={draw=black,very thick}]
    \path[-](A)edge node {} (B);
    \path[-](B) edge node {} (C);
    \path[-](D) edge node {} (C);
    \path[-](D) edge node {} (E);
    \path[-](E) edge node {} (G);
    \path[-](C) edge node {} (G);
    \path[-](B) edge node {} (G);
    \path[-](B) edge node {} (F);
    \path[-](C) edge node {} (F);
    \path[-](E) edge node {} (F);
\end{scope}
\end{tikzpicture}
\end{center}

Indeed $J(G)$ is quasi-equigenerated setting $d_{\alpha}(x_2)=2$ and $d_{\alpha}(x_i)=d_{\alpha}(y_j) = 1$ for all $j$ and $i \neq 2$.
\end{remark}


\section{Freiman cover ideals}

In this section we study Freiman property for quasi-equigenerated cover ideals by giving a general structure theorem and characterizing Freiman cover ideals among some classes of graphs. Our settings and notations, where not differently specified, will be the following: $I$ will be the cover ideal (quasi-equigenerated) of a graph $G$ on $n$ vertices or more generally, a quasi-equigenerated squarefree monomial ideal. We express the fiber cone of $I$ as the ring $$ F(I)= \dfrac{k[T_1, \ldots, T_t]}{\mathcal{I}},$$ where the variables $T_i$ correspond to minimal generators of $I$ via the usual ring homomorphism. We consider the minimal 1-covers of $G$ using the equivalent notation induced by the maximal independent sets described in Remark \ref{indipcover}. 
For $U$ a maximal independent of $G$, we call $T_U$ the correspondent variable in $k[T_1, \ldots, T_t]$.
We set $t=\mu(I)$ as the number of minimal generators of $I$ and $l(I)$ as its analytic spread. Often, we will use the notation $T_i, T_U$ also for their corresponding images in $F(I)$ instead of writing $\overline{T_i}, \overline{T_U}$. 

Set $a:=t-l(I)$ and call $b$ the number of minimal generators of $\mathcal{I}$ having degree $2$ with respect to the grading of the polynomial ring $k[T_1, \ldots, T_t]$.
The next lemma translates the Freiman condition in term of the invariants $a$ and $b$.

\begin{lem}  \label{freimanlemma}
Let $I$ be a squarefree quasi-equigenerated monomial ideal and let $\mathcal{I}$ be the defining ideal of the fiber cone $F(I)$. For $a,b$ defined as above, $b \leq \binom{a+1}{2}$ and $I$ is Freiman if and only if $b=\binom{a+1}{2}.$
 \end{lem}
\begin{proof}
Since $\mathcal{I}$ is generated by binomials, $\mu(I^2)= \binom{t+1}{2} - b$. Thus we have $$ \binom{t+1}{2} - b \geq (t-a)t - \binom{t-a}{2}
$$ and the equality holds if and only if $I$ is Freiman. A straightforward computation leads to our thesis.
\end{proof}

\begin{remark} \label{2generated}
In Theorem \ref{herzogthm} is stated that the fiber cone of a Freiman ideal has a 2-linear free resolution. An immediate consequence of this fact is that, when $I$ is Freiman, the ideal $\mathcal{I}$ is generated in degree 2.
\end{remark}

\begin{example} We observe the following facts:
\begin{enumerate}
    \item The cover ideals of complete graphs are Freiman of linear type. One can check this fact easily, but we shall prove a more general result in Theorem \ref{almostfreiman}.
    \item The graph in the following picture has cover ideal generated by $  f_1= x_1x_2x_4x_5, f_2= x_1x_3x_4, f_3= x_1x_3x_5, f_4= x_2x_3x_4, f_5= x_2x_3x_5.$ Its fiber cone is isomorphic to the ring $$ \dfrac{k[T_1,T_2,T_3,T_4,T_5]}{(T_2T_5-T_4T_3)} $$ and hence the cover ideal is Freiman by Lemma \ref{freimanlemma}.
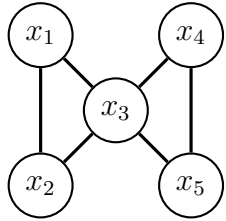
\begin{center}

\begin{tikzpicture}
\begin{scope}[every node/.style={fill=white,circle,thick,draw}]
    \node(A) at (-1,-1) {$x_2$};
    \node(B) at  (-1,1) {$x_1$};
    \node(C) at  (1,1) {$x_4$};
    \node(D) at (1,-1) {$x_5$};
    \node (G) at (0,0) {$x_3$};
\end{scope}
\begin{scope}
            [every edge/.style={draw=black,very thick}]
    \path[-](A)edge node {} (B);
    \path[-](B) edge node {} (G);
    \path[-](A) edge node {} (G);
    \path[-](C) edge node {} (G);
    \path[-](D) edge node {} (G);
    \path[-](D) edge node {} (C);
\end{scope}
\end{tikzpicture}
\end{center}

    
   

\item Call $H_3$ the graph in next picture, that is the graph on $6$ vertices $x_1, \ldots, x_6$ whose independent sets are $\lbrace x_1, x_2 \rbrace$, $\lbrace x_3, x_4 \rbrace$, $\lbrace x_5, x_6 \rbrace$, $\lbrace x_1, x_3 \rbrace$, $\lbrace x_2, x_5 \rbrace$, $\lbrace x_4, x_6 \rbrace.$ The fiber cone of this graph is isomorphic to the ring $$ \dfrac{k[T_1,T_2,T_3,T_4,T_5,T_6]}{(T_1T_2T_3-T_4T_5T_6)} $$ and hence the cover ideal is not Freiman by Remark \ref{2generated}.

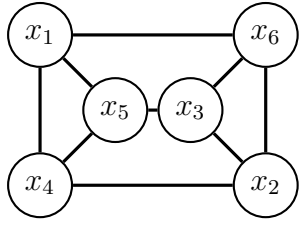
\begin{center}

\begin{tikzpicture}
\begin{scope}[every node/.style={fill=white,circle,thick,draw}]
    \node(A) at (-1,-1) {$x_4$};
    \node (G) at (0,0) {$x_5$};
    \node(B) at  (-1,1) {$x_1$};
    \node(C) at  (2,1) {$x_6$};
    \node(D) at (2,-1) {$x_2$};
   \node (E) at (1,0) {$x_3$};
\end{scope}
\begin{scope}
            [every edge/.style={draw=black,very thick}]
    \path[-](A)edge node {} (B);
    \path[-](B) edge node {} (G);
    \path[-](A) edge node {} (G);
    \path[-](C) edge node {} (E);
    \path[-](D) edge node {} (E);
    \path[-](D) edge node {} (C);
    \path[-](A)edge node {} (D);
    \path[-](B) edge node {} (C);
    \path[-](E) edge node {} (G);
\end{scope}
\end{tikzpicture}
\end{center}

\item Starting with $H_3$, one may inductively construct a family of graphs whose fiber cones have principal defining ideal generated by a binomial of degree $n$ for each $n \geq 3$
(clearly the cover ideals of the graphs of this family are not Freiman).
Indeed, for $n \geq 3$, define $H_{n+1}$ as the graph on the vertices $x_1, \ldots, x_{2n+2}$, having the same independent sets as $H_{n}$ except $\lbrace x_{2n-2}, x_{2n} \rbrace$ and having also the independent sets $\lbrace x_{2n+1}, x_{2n+2} \rbrace$, $\lbrace x_{2n-2}, x_{2n+1} \rbrace$, $\lbrace x_{2n}, x_{2n+2} \rbrace$.
\end{enumerate}
\end{example}

\medskip

As application of Lemma \ref{freimanlemma}, we can get a complete characterization of when the join of two graphs is Freiman.

\begin{thm} \label{freimanjoint}
Let $G_1$ and $G_2$ be two simple graphs and let $G:= G_1 \oplus G_2$. Let $I_1, I_2, I$ be respectively the cover ideals of $G_1, G_2, G$. Suppose $I_1$ and $I_2$ to be quasi-equigenerated. The following conditions are equivalent: \begin{enumerate}
\item $I$ is a Freiman ideal.
\item $I_1$ and $I_2$ are both Freiman ideals and at least one of them is of linear type.
\end{enumerate}
\end{thm}
 
 \begin{proof}
 Since $I_1$ and $I_2$ are quasi-equigenerated, so it is $I$ by Theorem \ref{quasilinked}.
Assume as in Definition \ref{joint}, $I_1=(f_1, \ldots, f_t) \subseteq k[x_1, \ldots, x_n]$ and $I_2=(g_1, \ldots, g_s) \subseteq k[y_1, \ldots, y_m]$. Write their fiber cones as $$ F(I_1) \cong \frac{k[T_1, \ldots, T_t]}{\mathcal{I}_1} $$ and $$ F(I_2) \cong \frac{k[U_1, \ldots, U_s]}{\mathcal{I}_2} .$$ We claim that $$ F(I) \cong \frac{k[T_1, \ldots, T_t, U_1, \ldots, U_s]}{F_2\mathcal{I}_1 + F_1\mathcal{I}_2}, $$ where $F_1=x_1x_2\cdots x_n$, $F_2=y_1y_2\cdots y_m$. Observe that by Proposition \ref{linked}, the generators of $I$ are of the form $f_iF_2$ and $g_jF_1$.  Hence, using the usual homomorphism defined by $T_i \to f_iF_2$ and $U_j \to g_jF_1$, we identify $F(I)$ with the quotient of $ k[T_1, \ldots, T_t, U_1, \ldots, U_s] $ by a defining ideal $ \mathcal{I} $. Moreover, notice that if $q \in \mathcal{I}_1$, then $F_2q \in \mathcal{I}$ and analogously if $q \in \mathcal{I}_2$, then $F_1q \in \mathcal{I}$. We only need to show that these are all the generators of $\mathcal{I}$. To do this, let $$ p= T_{i_1} \cdots T_{i_{r_1}}U_{j_1}\cdots U_{j_{r_2}} - T_{k_1}\cdots T_{k_{r_3}}U_{l_1}\cdots U_{l_{r_4}} \in \mathcal{I} $$ (observe that by Proposition \ref{defidealhomog}, $r_1+r_2=r_3+r_4$), hence we must have $$ f_{i_1} \cdots f_{i_{r_1}} F_2^{r_1} g_{j_1}\cdots g_{j_{r_2}} F_1^{r_2} = f_{k_1}\cdots f_{k_{r_3}} F_2^{r_3} g_{l_1}\cdots g_{l_{r_4}}F_1^{r_4}.  $$ Separating the variables, this implies $ f_{i_1} \cdots f_{i_{r_1}} F_1^{r_2} = f_{k_1}\cdots f_{k_{r_3}} F_1^{r_4} $. Assuming by way of contradiction $r_2 < r_4$ (or analogously $r_4 < r_2$), we get $ f_{i_1} \cdots f_{i_{r_1}} = f_{k_1}\cdots f_{k_{r_3}} F_1^{r_4-r_2} $ and, since $I_1$ is quasi-equigenerated, we can find $\alpha \in \mathbb{N}_{\geq 1}^{n}$ such that $$ r_1d_{\alpha}(f_{i_1})= r_3d_{\alpha}(f_{i_1})+(r_4-r_2)r_1d_{\alpha}(F_1).$$ But this is a contradiction since $r_4-r_2=r_1-r_3$ and since $f_{i_1}$ properly divides $F_1$. Therefore we must have $r_2=r_4$, $r_1=r_3$, and hence $p$ is in the ideal generated by other generators of $ \mathcal{I} $ of the form $F_2q $ and $F_1q $ for $q \in \mathcal{I}_1$ or $q \in \mathcal{I}_2$ and this proves our claim. 
\\
Now, it is easy to observe that the analytic spread $l(I)=l(I_1)+l(I_2)$. Following the notation of Lemma \ref{freimanlemma}, denote by $b_1, b_2, b$ the number of minimal generators of of degree $2$, respectively of $\mathcal{I}_1$, $\mathcal{I}_2$, $\mathcal{I}$. Also write $a_1:=n -l(I_1)$, $a_2:=m -l(I_2)$, $a:=n+m -l(I)$. Hence, clearly $b=b_1+b_2$ and $a=a_1+a_2$. Thus $$ b \leq \binom{a_1+1}{2} + \binom{a_2+1}{2} \leq \binom{a+1}{2}. $$ The first inequality is an equality if and only if both $I_1$ and $I_2$ are Freiman ideals, while the second one is an equality if and only if one among $a_1$ and $a_2$ is zero, meaning that $ F(I_1) $ or $F(I_2)$ is a polynomial ring.
By Lemma \ref{freimanlemma}, $I$ is a Freiman ideal if and only if both these conditions hold.
\end{proof}

\medskip

In \cite[Theorem 1.3]{HHZ}, it is shown that a quasi-equigenerated ideal $I$ is Freiman if and only if $F(I)$ has minimal multiplicity. We want to use this fact, together with the classification of homogeneous domain (among quotients of polynomial rings) of minimal multiplicity given in \cite[Section 4]{EG}, in order to describe Freiman squarefree quasi-equigenerated monomial ideals as ideals generated by minors of certain matrices.

\begin{lem} \label{2xm}
Let $\mathbf{T}=\{T_1,\dots,T_s\}$ for some $s>4$ be a set of indeterminants, and let $\mathbf{A},\mathbf{B},\mathbf{C},\mathbf{D}\subset\mathbf{T}$ such that $T_1\not\in\mathbf{A}$, $T_2\not\in\mathbf{B}$, $T_3\not\in\mathbf{C}$, and $T_4\not\in\mathbf{D}$.  Let $\ell_1\in k[\mathbf{A}]$, $\ell_2\in k[\mathbf{B}]$, $\ell_3\in k[\mathbf{C}]$, and $\ell_4\in k[\mathbf{D}]$ be linear forms.  
Suppose that the matrix

\[ M=\left(\begin{matrix} 
aT_1+\ell_1 & bT_2+\ell_2 \\
cT_3+\ell_3 & dT_4+\ell_4 
\end{matrix}\right)\]
satisfies $\det M = \alpha (T_1T_4-T_2T_3)$ for some $\alpha \in k$.  Then $\ell_1=\ell_2=\ell_3=\ell_4=0$.
\end{lem}
\begin{proof}
Since $\ell_1\in k[\mathbf{A}]$, $\ell_1=\sum_{A\in\mathbf{A}}a_AA$ where each $a_A\in k$.  Similarly $\ell_2=\sum_{B\in\mathbf{B}}b_BB$, $\ell_3=\sum_{C\in\mathbf{C}}c_CC$, and $\ell_4=\sum_{D\in\mathbf{D}}d_DD$ where each $b_B,c_C,d_D\in k$.\\
We may assume that $\mathbf{A}\cap\mathbf{D}=\emptyset=\mathbf{B}\cap\mathbf{C}$.  Indeed, suppose that $\mathbf{A}\cap\mathbf{D}$ is nonempty.  Then there is some term $A\in\mathbf{A}$ also contained in $\mathbf{D}$.  Since the terms of $\det M$ are squarefree, we conclude that $A\in\mathbf{B}\cap\mathbf{C}$ as well.  Thus, by applying a row operation to $M$, we can obtain an new matrix, \[M'=\left(\begin{matrix} 
aT_1+\ell_1 & bT_2+\ell_2 \\
cT_3+\ell_3' & dT_4+\ell_4' 
\end{matrix}\right)\]
with the same determinant as $M$ but such that $\ell_1$ and $\ell_4$ share one fewer variables than in $M$.  Repeating this process we can eliminate all elements of $\mathbf{A}\cap\mathbf{D}$.  Working symmetrically, the same conclusion holds for $\mathbf{B}\cap\mathbf{C}$.

We also may assume that $a=d=1$, since scaling rows of matrices scales determinants by constants.  

 Suppose $\ell_4\neq 0$. Since $\ell_4$ is nonzero then $(T_1+\ell_1)(T_3+\ell_4)$ contains terms $T_1D_i$ for each $1\leq i\leq v$.  As these terms do not appear in the determinant, they must be cancelled by terms of $(bT_2+\ell_2)(cT_3+\ell_3)$.  Thus $T_1\in\mathbf{B}\cup\mathbf{C}$, and so  $\mathbf{D}\subseteq\mathbf{B}$.  Since $T_1\in\mathbf{C}$, $(bT_2+\ell_2)(cT_3+\ell_3)$ contains a term of the form $T_1B$ for each $B\in\mathbf{B}$.  Since these terms do not appear in the determinant, they must be cancelled with terms from the main diagonal.  Thus $\mathbf{D}\subseteq\mathbf{B}$ and so $\mathbf{B}=\mathbf{D}$.  Since terms of the form $a_AT_4A$ occur on the product of the main diagonal, but not in the determinant, they must cancel with terms of $(bT_2+\ell_2)(cT_3+\ell_3)$.  Thus $T_4\in\mathbf{B}\cup\mathbf{C}$.  Since $T_4\not\in\mathbf{D}=\mathbf{B}$, $T_4\in \mathbf{C}$.  However, this means that $\alpha T_4B$ is a term of the product of the antidiagonal which doesn't appear in the determinant and thus must be cancelled with a term from $(T_1+\ell_1)(T_3+\ell_4)$.  This cannot happen because $B\in\mathbf{D}$ and thus $B\not\in\mathbf{A}$.\\
 Thus $\ell_4=0$ and by symmetric arguments, $\ell_1=\ell_2=\ell_3=0$.
\end{proof}

 \begin{thm} Let $I$ be a squarefree quasi-equigenerated monomial ideal. Then $I$ is a Freiman ideal if and only if the defining ideal of its fiber cone $\mathcal{I}$ is either $(0)$ or it is generated by the minors of a generic $2\times m$ matrix. 
 \label{matrix}
 \end{thm}
 
 \begin{proof} 
 When $\mathcal{I}=(0)$ or when it is generated by the minors of a $2\times m$ matrix, by the results in \cite[Section 4]{EG}, $F(I)$ has minimal multiplicity and hence $I$ is Freiman.

Conversely assume $F(I)$ to have minimal multiplicity and $\mathcal{I} \neq (0)$.
By \cite[Theorem 4.3]{EG}, and since $\mathcal{I}$ is generated by binomials, either $\mathcal{I}$ is the ideal generated by the minors of a $2\times m$ matrix of linear forms or is the ideal generated by the $2\times 2$ minors of a generic $3\times 3$ symmetric matrix of linear forms (notice that the case in which $F(I)$ is an hypersurface generated by a quadratic binomial is included in the first case). But, since $I$ is squarefree, for any choice of three monomials $f,g,h$ among the minimal generators of $I$, $f^2 \neq gh$, and thus the case of minors of a symmetric $3\times 3$ matrix has to be excluded.

Suppose that $I=I_2(M)$ for some $2\times m$ matrix $M$, of linear forms.  Since $I$ is generated by squarefree binomials of degree 2, we may assume by Lemma \ref{2xm} that the entries of $M$ are monomials, and hence they are generic.
\end{proof}
 
\medskip

Next result will be useful in the following to study the Freiman property for the cover ideals of some classes of graphs.

 \begin{prop}
  \label{primelements}
  Let $G$ be a graph and $I$ its cover ideal. Consider a vertex $x$ belonging to only one maximal independent set $U$ of $G$. Then, the variable $T_U$, associated to the minimal 1-cover $h_U$, is a prime element in the fiber cone $F(I).$
 \end{prop}
 
\begin{proof}
Consider the defining ideal of the fiber cone $\mathcal{I}=(q_1, \ldots, q_s)$. By Proposition \ref{quotients}, the variable $T_U$ is a prime element in $F(I)$ if and only if it does not divide any monomial of the binomials $q_1, \ldots, q_s$. Let $q=\alpha-\beta$ be a minimal generator of $ \mathcal{I}$ with $\alpha, \beta$ monomials, and assume by way of contradiction $T_U$ divides $\alpha$. By Proposition \ref{defidealhomog}, $\mathcal{I}$ is an homogeneous ideal and it follows that there exist some minimal 1-covers $h_i,h_j$ of $G$ and an integer $u \geq 1$ such that $$ h_U^u h_{i_1} \ldots h_{i_r} = h_{j_1} \ldots h_{j_{r+u}}.
$$
But, since $x$ is contained in only one independent set, then it divides every minimal 1-cover of $G$ different from $h_U$ and this induces a contradiction since the degree of $x$ in 
$h_U^u h_{i_1} \ldots h_{i_r}$ is now strictly less than the degree of $x$ in $ h_{j_1} \ldots h_{j_{r+u}}.$
\end{proof}

 An easy application of last proposition shows that the elements of a family of graphs including complete graphs have cover ideal of linear type.

 \begin{definition}
A graph of $n+1$ vertices is said to be \it almost complete \rm if it has an induced complete subgraph of $n$ vertices. An almost complete graph having vertices $x_1, x_2, \ldots, x_{n} , y$ is of the form $$ G= K_{n} \oplus_{U, \lbrace y \rbrace} \lbrace y \rbrace $$ where $ K_{n} $ is the complete graph on the vertices $x_1, x_2, \ldots, x_{n}$ and $U$ is a subset of this same set of vertices.
 \end{definition}
 
 \begin{remark}
 A graph $G$ is complete if and only if $c(G)=1$ and it is almost complete if and only if $c(G)=2$ and the intersection of all the independent sets containing two elements is non-empty. 
 
Let $G$ be an almost complete graph with $n \geq 4$ vertices and let $I$ be its cover ideal.
Let $d= n - |\mathcal{N}(y)|$, where $\mathcal{N}(y)$ is the set of neighbors of $y$. 
\begin{enumerate}
\item If $d=0$, $G = K_{n+1}$.
\item If $d=1$, there exists a unique vertex $x_i$ not adjacent to $y$, and $\N(x_i)=\N(y)$. Thus, by Theorem \ref{sdefequal}, $I$ has fiber cone isomorphic to the fiber cone of $J(K_n)$. 
\end{enumerate} 
\end{remark}

\begin{thm} \label{almostfreiman}
Let $G=(V,E)$ be an almost complete graph with $n \geq 4$ vertices and let $I$ be its cover ideal. Then $I$ is a quasi-equigenerated Freiman ideal of linear type.
\end{thm}
 
 \begin{proof}
Let $U\subseteq V$ be the set of neighbors of $y$. 
Set $F:=x_1x_2 \cdots x_ny$. It is easy to check that the cover ideal $I$ of an almost complete graph $G=(V,E)$ is generated by the monomials $h_j= F x_j^{-1}$ for $x_j \in U$, correspondent to the maximal independent set $\lbrace x_j \rbrace$ and $g_l= Fx_l^{-1}y^{-1}$ for $x_l \in V \setminus U$, correspondent to the maximal independent set $\lbrace x_l, y \rbrace$. In particular $c(G)=2$ and therefore $I$ is quasi-equigenerated by Proposition \ref{quasic2}.

By Proposition \ref{primelements}, all the minimal 1-covers of $G$ are prime elements in the fiber cone of $I$. It follows that $I$ is of linear type.
\end{proof}

\section{Families of graphs and Freiman cover ideals}

In this section, we consider some different families of graphs having quasi-equigenerated cover ideal and we characterize when their cover ideal are Freiman. We also provide explicit computations of their relations and analytic spreads.

\subsection{Pairs of complete graphs sharing a vertex}

\begin{definition} \label{twocompletedef}
Let $m \geq n \geq 2 $. We define the graph $A_{n,m}=K_{n}\oplus_{\{x_n\},V(K_{m-1})}K_{m-1}$ as the graph having vertex set $$ V=\lbrace x_1, x_2, \ldots, x_n, y_2, y_3, \ldots, y_m \rbrace $$ and edge set given by the union $$ E= \bigcup_{1 \leq i < j}^{n} (x_i, x_j) \cup \bigcup_{2 \leq i<j}^{m} (y_i, y_j) \cup \bigcup_{j=2}^{m} (x_n, y_j).   $$ 
\end{definition}

\begin{thm} \label{twocomplete}
Let $m \geq n \geq 2 $ and let $I$ be the cover ideal of the graph $A_{n,m}$. Then, $I$ is quasi-equigenerated and:
\begin{enumerate}
\item The analityc spread of $I$ is $n+m-2.$
\item $I$ is a Freiman ideal if and only if $n \leq 3$.
\end{enumerate}
\end{thm}
 
 \begin{proof}
Observe that $c(G)=2$ and hence $I$ is quasi-equigenerated by Proposition \ref{quasic2}. If $n=2$, the graph is almost complete and therefore $I$ is Freiman of linear type by Theorem \ref{almostfreiman}. 
Let $F=x_1x_2\cdots x_ny_2y_3\cdots y_m$. The minimal generators of $I$ are $Fx_n^{-1}$ and the others are of the form $Fx_i^{-1}y_j^{-1}$ for every possible choice of $1\leq i \leq n-1$ and $2\leq j \leq m$. Here we use a slightly different notation for the fiber cone (with respect to Section 4), associating $T $ to $ Fx_n^{-1}$ and $T_{i,j} $ to $ Fx_i^{-1}y_j^{-1}$, and writing $F(I)$ as the quotient of the polynomial ring $ k[T, T_{1,2}, \ldots, T_{1,m}, T_{2,2}, \ldots, T_{n-1,m}] $ by a homogeneous binomial ideal $ \mathcal{I} $.

Let $q=\alpha-\beta \in \mathcal{I}$ where $\alpha, \beta$ are monomials. Observe that by Proposition \ref{primelements}, the image of $T$ is a prime element of $F(I)$ and $T$ divides $\alpha$ if and only if divides also $\beta$. 
Therefore, assuming $q$ part of a set of minimal generators, we get that $\alpha, \beta$ are of the form $ T_{i_1,j_1}T_{i_2,j_2} \cdots T_{i_r,j_r} $. Moreover, notice the following important facts: 
\begin{enumerate}
\item[$(\ast)$] $T_{i,j}$ divides $\alpha$ if and only if there exists $k \neq j$ such that $T_{i,k}$ divides $\beta.$
\item[$(\ast \ast)$] For every $i_1 \neq i_2, j_1 \neq j_2$, the binomial $T_{i_1,j_1}T_{i_2,j_2} - T_{i_1,j_2}T_{i_2,j_1}$ is a minimal generator of $\mathcal{I}$.
\end{enumerate}
If $\deg(q)=2$ and $\alpha=T_{i,j}T_{i,k}$ for the same $i$, then $\beta$ is clearly forced to be equal to $\alpha$, making $q=0$. Hence all the minimal generators of $\mathcal{I}$ of degree $2$ are of the form described in $(\ast \ast)$. The number of these generators is $$ b= \binom{n-1}{2}\binom{m-1}{2}. $$
Now, we compute the analytic spread $l(I)$ of $I$. With an abuse of notation we write also the elements of $F(I)$ in the form $T,T_{i,j}$ instead of writing their classes modulo $\mathcal{I}$. Define for $i=1, \ldots, n-2$ the ideals: 
$$ \mathcal{I}_i:= (T_{i,2}, T_{i,3}, \ldots, T_{i,m}) \subseteq F(I) $$ and $$\mathcal{P}_i= \sum_{k=1}^{i}\mathcal{I}_i \subseteq F(I). $$
The fact $(\ast)$ together with Proposition \ref{quotients}, shows that each ideal $ \mathcal{P}_i $ is prime. Moreover we want to prove by induction that the height of $\mathcal{P}_i$ is $i$. 
For $i=1$, consider the localization $F(I)_{\mathcal{P}_1}$ and observe that $T$ and all the elements $T_{i,j}$ with $i \neq 1$ are units in this ring. By the fact $(\ast\ast)$, it is easy to see that, for every $j,k$, $T_{1,j}$ and $T_{1,k}$ are associated in $F(I)_{\mathcal{P}_1}$, and hence $F(I)_{\mathcal{P}_1} \cong k[T_{1,2}]_{(T_{1,2})}$ has dimension one. Assume now for some $i \leq n-2$ that $\mathcal{P}_{i-1}$ has height $i-1$ and show that $\mathcal{P}_{i}$ has height $i$. The ideal $ J := \frac{\mathcal{P}_{i}}{\mathcal{P}_{i-1}} $ is prime in the ring $ R:= \frac{F(I)}{\mathcal{P}_{i-1}} $ and the binomials of the form $T_{i,j}T_{l,k} - T_{i,k}T_{l,j} \in \mathcal{I} \setminus \mathcal{P}_{i-1}$ for every $j,k$ and $i < l \leq n-1$. Hence, again by the fact $(\ast\ast)$, $T_{i,j}$ and $T_{i,k}$ are associated in $R_{J}$, and hence, as in the previous case, $R_{J}$ has dimension one, implying $F(I)_{\mathcal{P}_i}$ to have dimension $i$.
Finally, observe that $$ \frac{F(I)}{\mathcal{P}_{n-2}} \cong k[T, T_{n-1,2}, \ldots, T_{n-1,m}]  $$ is a polynomial ring in $m$ variables and therefore $l(I)= n-2+m.$ Set $$ a := \mu(I)-l(I)= ((n-1)(m-1)+1) - (n+m-2) = (n-2)(m-2). $$
By Lemma \ref{freimanlemma}, comparing $b$ with $\binom{a+1}{2}$, we get $$ \binom{n-1}{2}\binom{m-1}{2} \leq \binom{(n-2)(m-2)+1}{2} $$ and $I$ is a Freiman ideal if and only if the equality holds. If $n=2,3$, the equality holds, otherwise for $m,n \geq 4$, a straightforward computation shows that the inequality is strict and therefore $I$ is not Freiman in such cases.
\end{proof}

\begin{remark}
Following the notation of the preceding theorem, the defining ideal of the fiber cone of the cover ideal of $A_{n,m}$ can be seen to be generated by the minors of the $n \times m$ matrix $$ \begin{pmatrix}
T_{i_1,j_1} & T_{i_1,j_2} & \ldots & T_{i_1,j_m}  \\ 
T_{i_2,j_1} & T_{i_2,j_2} & \ldots & T_{i_2,j_m} \\
\vdots & \vdots & \vdots & \vdots \\
T_{i_n,j_1} & T_{i_n,j_2} & \ldots & T_{i_n,j_m} 
\end{pmatrix}.  $$ This gives another argument to characterize when these ideals are Freiman using Theorem \ref{matrix}.
\end{remark}

\subsection{Circulant graphs}

 In Theorem \ref{quasicirculant} we classified which circulant graphs have (quasi-)equigenerated cover ideals. Here we describe which of their cover ideals are Freiman. We recall that given a graph $G$ and a vertex $x_i \in V(G)$, we call $G_i$ the induced subgraph of $G$ on the set $V(G) \setminus (\mathcal{N}(x_i) \cup \{ x_i \})$.


\begin{thm}
\label{CircFrei}
Let $G=C_n(1,\ldots,s)$. The cover ideal $J(G)$ is Freiman if and only if $s>\frac{n-4}{2}$ or $G\in\{C_5,C_7\}$.
\end{thm}
\begin{proof}
Both $J(C_5)$ and $J(C_7)$ are of linear type and thus are Freiman. Since we only want to consider graphs with equigenerated cover ideal, we suppose $s\geq 2$.
First assume $s\leq\frac{n-4}{2}$.  Then $G_1$ has vertex set $\{x_{s+1},x_{s+2},\dots,x_{n-s}\}$ of size at least $3$ and $x_{n-s},x_{n-s-1},x_{n-s-2}$ are contained in a clique of $G_1$.  There exists a maximal independent set $H$ in $G_1$ with $x_{n-s-2}\in H$.

 Let $\phi:V\rightarrow V$ be defined by $\phi(x_i)=x_{i+1}$ (modulo $n$).  Note that $\phi$ preserves adjacency of vertices. Since $x_{n-s-2}\in H$ and $s\geq 2$, $x_{n-s-1},x_{n-s}\not\in H$, hence $H_1=\{x_1\}\cup H$, $H_2=\{x_1\}\cup\phi(H)$, and $H_3=\{x_1\} \cup \phi^2(H)$ are all independent sets of $G$.  Similarly, we see that $H_4=\{x_2\}\cup\phi(H)$,  $H_5=\{x_2\} \cup \phi^2(H)$, $H_6=\{x_2\}\cup\phi^3(H)$, $H_7=\{x_3\} \cup \phi^2(H)$, and $H_8=\{x_3\}\cup\phi^3(H)$ are all maximal independent sets of $G$.  However we see then that 
 \[h_{H_2}h_{H_6}h_{H_7}-h_{H_3}h_{H_4}h_{H_8}\]
is a minimal generator of $\mathcal{I}$ of degree 3.  Thus $J(G)$ is not Freiman by Remark \ref{2generated}.

Suppose that $s>\frac{n-4}{2}$.  Then either $|G_i|=2$, $|G_i|=1$, or $|G_i|=0$ for all $i$.  If each $|G_i|=0$, then $G$ is complete and $J(G)$ Freiman.  If each $|G_i|=1$ then $J(G)$ is generated by elements of the form 
$F(x_ax_{a+s+1})^{-1}$.  For each $x_a$, there is exactly one element of this generating set not divisible by $x_a$, and thus $J(G)$ is Freiman by Proposition \ref{primelements}.  

If each $|G_i|=2$ then $J(G)$ is generated by elements of the form $h_{a,b}=F(x_ax_b)^{-1}$ where $b\in\{a+s+1,a+s+2\}$. As usual we consider $T_{a,b}$ to be the image of $h_{a,b}$ in the fiber cone of $J(G)$. Let $\alpha - \beta$ be a binomial minimal generator of $\mathcal{I}$, where $\alpha, \beta$ are monomials. If $T_{a,a+s+1}$ divides $\alpha$, then $T_{a,a+s+2}$ divides $\beta$, since $T_{a,a+s}$ and $T_{a,a+s+2}$ are the only generators of $J(G)$ which $x_a$ does not divide.  Since the highest power of $x_{a+s+1}$ dividing $T_{a,a+s+1}$ is one less than the highest power dividing $T_{a,a+s+2}$, there must exist another $T_{c,d}$ dividing $\beta$ and not divisible by $x_{a+s+1}$.  As $|G_{a+s+1}|=2$, we know that the only such element is $T_{a+s+1,a-1}$.  But then the highest power of $x_{a-1}$ dividing $\beta$ is one lower than the highest power dividing $\alpha$ unless another generator of $J(G)$ which is not divisible by $x_{a-1}$ divides $\alpha$.  Continuing this process, we see that all the generators of $J(G)$ must divide both $\alpha$ and $\beta$, and therefore the fiber cone is a  polynomial ring.  Thus $J(G)$ is Freiman.
\end{proof}

\subsection{Whiskered graphs}

 \begin{definition}
Given a graph $G$ with $V(G)=\{x_1,\dots,x_n\}$, the whiskering of $G$ is the graph $\widetilde{G}$ with $V(\widetilde{G})=V(G)\cup\{y_1,\dots,y_n\}$ and $E(\widetilde{G})=E(G)\cup\{\{x_i,y_i\}|i=1,\dots,n\}$. The graph in next picture is the whiskering of $C_3.$
 \end{definition}
 
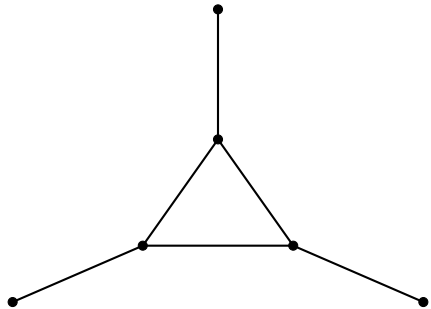
\begin{center}
     \begin{tikzpicture}
     \draw[fill] (-1,0) circle [radius=.06];
     \draw[fill] (1,0) circle [radius=.06];
     \draw[fill] (0,1.414) circle [radius=.06];
     \draw[fill] (0,3.144) circle [radius=.06];
     \draw[fill](2.73,-.75) circle [radius=.06];
     \draw[fill] (-2.73,-.75) circle [radius=.06];
     \draw[thick] (-1,0)--(1,0);
     \draw[thick] (-1,0)--(0,1.414);
     \draw[thick] (1,0)--(0,1.414);
     \draw[thick] (0,1.414)--(0,3.144);
     \draw[thick] (-1,0)--(-2.73,-.75);
     \draw[thick] (1,0)--(2.73,-.75);
     \end{tikzpicture}
 \end{center}

\begin{definition}
Given a graph $G$, let $I=J(G)$ be its cover ideal. For any squarefree 1-cover $f \in I$ (non necessarily minimal), we define $A_f = \{i \, | \, x_i \mbox{ does not divide } f\} $ and $$ f_w:= f \prod_{i \in A_f} y_i.$$
 \end{definition}

\begin{prop} \label{genwhisk}
Let $G$ be a graph on $n$ vertices and let $\widetilde{G}$ be its whiskering.  Then $J(\widetilde{G})$ is minimally generated by the set $$ \{f_w \, | \, f \in J(G) \mbox{ squarefree 1-cover} \} $$ and it is equigenerated of degree $n$.
\end{prop}
\begin{proof}
Clearly, by definition, $f_w$ is a 1-cover of $\widetilde{G}$. To show that they are all minimal covers of $\widetilde{G}$, take $f,h \in J(G)$ squarefree 1-covers. Hence, we may find $x_i$ dividing $f$ and not $h$ (or we may find the opposite case). It follows that $x_i$ divides $f_w$ and not $h_w$, while $y_i$ divides $h_w$ and not $f_w$, and therefore $f_w$ and $h_w$ are both minimal. Obviously each $f_w$ has degree $n$, making $J(\widetilde{G})$ equigenerated in degree $n$.
\end{proof}

 Let $G$ be a graph on $n$ vertices and let $J(\widetilde{G})$ be the cover ideal of the whiskering of $G$. We make use of the following notation, using Proposition \ref{genwhisk}, to describe the minimal generators of $J(\widetilde{G})$. Since $F=x_1x_2 \cdots x_n$ is a 1-cover of $G$, then $F=F_w$ is also a minimal 1-cover of $ \widetilde{G} $. The others generators of $J(\widetilde{G})$ are of the form $$ h_U := (F \prod_{x_i \in U} x_i^{-1})_w $$ where $U$ is an independent set of $G$. In the case $U=\lbrace x_i \rbrace$, we denote $$ h_i := (F x_i^{-1})_w. $$ By these observations, it follows that $\mu(J(\widetilde{G}))= n + 1 + d$ where $d$ denotes the number of independent sets of $G$ of cardinality at least two.

\begin{thm} \label{analitycspread}
Let $G$ be a graph on $n$ vertices and let $I=J(\widetilde{G})$ be the cover ideal of the whiskering of $G$. The analytic spread of $I$ is $l(I)= n+1.$
\end{thm}

\begin{proof}
We proceed along the line of the computation of the analytic spread done in the proof of Theorem \ref{twocomplete}.
Let us associate the variable $T $ to $F$ and, for any independent set $U$ of $G$, the variable $T_U$ to the 1-cover $h_U$ (for $U= \lbrace x_i \rbrace$ we call the variables $T_{U}:= T_i $).
Thus, we present $F(I)$ as the quotient of the polynomial ring $ k[T, T_{1}, \ldots, T_{n}, T_{U_1}, \ldots, T_{U_d}] $ by a homogeneous binomial ideal $ \mathcal{I} $. The generators of $\mathcal{I}$ are of the form \begin{equation} \label{1}
T_{U_{l_1}}T_{U_{l_2}} \cdots T_{U_{l_r}}T^c - T_{U_{j_1}} T_{U_{j_2}} \cdots T_{U_{j_s}}
\end{equation} where $c +r=s$, $U_{l_v},U_{j_k}$ are independent sets of $G$, and $U_{l_v} \neq U_{j_k}$ for every $l_v, j_k$. With an abuse of notation we write also the elements of $F(I)$ in the form $T_{U}$ instead of writing their classes modulo $\mathcal{I}$.
 Define now for $i=1, \ldots, n$ the ideals $$ \mathcal{I}_i:= (T_U \, | \, U \mbox{ is an independent set of } G \mbox{ and } x_i \in U ) \subseteq F(I). $$ Notice that $\mathcal{I}_i$ always contains $T_i$ and it is principal if and only if $x_i$ is connected to any other vertex of $G$. We define, for $i=1, \ldots, n$, $$\mathcal{P}_i= \sum_{k=1}^{i}\mathcal{I}_k $$ and $$\mathcal{P}_{n+1}= \mathcal{P}_{n} + (T),$$ and we observe that since any minimal 1-cover of $\widetilde{G}$ belongs to some of the ideals $ \mathcal{I}_i $, the homogeneous maximal ideal of $F(I)$ is equal to $ \mathcal{P}_{n+1} $. We need to show that $ \mathcal{P}_i $ is a prime ideal of $F(I)$ of height $i$ and this will imply our thesis. 
 
 Consider a binomial $q$ minimal generator of $\mathcal{I}$. Clearly, since $q$ is of the form given in (\ref{1}), $x_i \in U_{l_v}$ for some $v$ if and only if $x_i \in U_{j_k}$ for some $k$, hence either $q \in \mathcal{P}_i$ 
 or all the independent sets $U_{l_v},U_{j_k}$ do not contain the vertices $x_1, \ldots, x_i$. By Proposition \ref{quotients}, it follows that for every $i$, $\mathcal{P}_i$ is a prime ideal. \\
We prove by induction that the height of $\mathcal{P}_i$ is $i$. Consider the localization $F(I)_{\mathcal{P}_1}$ and observe that $T$ and all the elements $T_U$ such that $x_1 \not \in U$ are units in this ring. Since, if $\lbrace x_1 \rbrace \subsetneq U$, the binomial $T_UT-T_1T_{U \setminus \lbrace x_1 \rbrace} \in \mathcal{I}$, then $T_1$ is associated to $T_U$ in $F(I)_{\mathcal{P}_1}$, and hence $F(I)_{\mathcal{P}_1} \cong k[T_1]_{(T_1)}$ has dimension one. Applying the same argument inductively, as done in the proof of Theorem \ref{twocomplete}, we get that $F(I)_{\mathcal{P}_i} \cong k[T_1,\ldots, T_i]_{(T_1,\ldots, T_i)}$ has dimension $i$ and therefore $ \mathcal{P}_i $ has height $i$.
\end{proof}

 
\begin{thm} \label{cg2}
Let $G$ be a graph on $n$ vertices. 
Let $I=J(\widetilde{G})$ be the cover ideal of the whiskering of $G$. Then $I$ is Freiman if and only if $G$ is almost complete.
\end{thm}

\begin{proof}
We keep the same notation used the proof of Theorem \ref{analitycspread} for the fiber cone $F(I)$ and its defining ideal $\mathcal{I}$. Let $b$ be the number of minimal generators of $\mathcal{I}$ of degree $2$, and let $d$ be the number of independent sets (not necessarily maximal) of $G$ of cardinality at least $2$. By Theorem \ref{analitycspread}, we have $\mu(I)-l(I)=(n+1+d)-(n+1)= d$, and therefore by Lemma \ref{freimanlemma}, $$b \leq \binom{d+1}{2}= d + \binom{d}{2}$$ and $I$ is Freiman if and only if the equality holds. 

First assume $c(G)= w \geq 3$. Let $U=\lbrace x_{i_1}, x_{i_2}, \ldots, x_{i_w} \rbrace$ be an independent set of $G$ of maximal cardinality. Observe that \begin{equation} \label{eq1.5}
T_UT^{w-1}- T_{i_1} T_{i_2}\cdots T_{i_w}
\end{equation}
is a minimal generator of $\mathcal{I}$ of degree greater than 2. It follows by Remark \ref{2generated}, that $I$ is not Freiman.

Assume now $c(G)=2.$ Hence $d$ denotes the number of independent sets of $G$ containing exactly two elements. If $U=\lbrace x_i, x_j \rbrace$ is one of these sets, we have that \begin{equation} \label{eq2}
T_UT- T_i T_j
\end{equation}  is a minimal generator of $\mathcal{I}$, and there are exactly $d$ minimal generators of this form. Given two independent sets $U_1=\lbrace x_i, x_j \rbrace$ and $U_2=\lbrace x_i, x_k \rbrace$ sharing the vertex $x_i$, we get that also \begin{equation} \label{eq3}
T_{U_1}T_k - T_{U_2} T_j
\end{equation} is a minimal generator of $\mathcal{I}$. Observe that $G$ is almost complete if and only if all the independent sets containing two elements share the same fixed vertex. In this case, there are exactly $\binom{d}{2}$ minimal generators of $\mathcal{I}$ of this second form, implying $b = d + \binom{d}{2}$ and completing part of the proof. Naturally in this case, calling $x_i$ the common vertex of all the independent sets, the binomials in (\ref{eq2}) and (\ref{eq3}) are the minors of the $2 \times (d+1)$ matrix: $$ \begin{pmatrix}
T & T_1 & \ldots & T_d  \\ 
T_i & T_{U_1} & \ldots & T_{U_d}
\end{pmatrix}.  $$

In the other case, assuming $G$ not almost complete, there must be two independent sets $U_1,U_2$ of cardinality 2, such that $U_1 \cap U_2 = \emptyset$, indeed, if by way of contradiction this does not happen, there must exist three independent sets of the form $ \lbrace x_i, x_j \rbrace, \lbrace x_j, x_k \rbrace, \lbrace x_k, x_i \rbrace $. But this would imply that $ \lbrace x_i, x_j, x_k \rbrace $ is an independent set, contradicting the fact that $c(G)=2$. 

Now, set $U_1=\lbrace x_{i_1}, x_{j_1} \rbrace$ and $U_2=\lbrace x_{i_2}, x_{j_2} \rbrace$ and assume $U_1 \cap U_2 = \emptyset$. Clearly, the two binomials $T_{U_1}T- T_{i_1} T_{j_1}$ and $T_{U_2}T- T_{i_2} T_{j_2}$ cannot be minors of the same $2\times m$
 matrix, and therefore $I$ is not Freiman by Theorem \ref{matrix}.
\end{proof}

 \begin{cor} 
\label{freimantrees} Let $T$ be a reduced tree.
 The cover ideal of $ T $ is Freiman if and only if $T=P_2,P_4, \widetilde{P_3}$.
\end{cor}

 \begin{proof}
 Apply Theorem \ref{Trees} and Theorem \ref{cg2}.
 \end{proof}

\section*{Acknowledgements}
The second author has received financial support from Indam (Istituto Nazionale di Alta Matematica) to spend three months of Spring 2019 at Ohio State University. This article has been carried out during this period of time and the author acknowledges the support of Indam for giving him this opportunity of research. Also special thanks are addressed to Professor Alexandra Seceleanu for inviting him to visit University of Nebraska, and for the interesting conversations about the content of this work.


\end{document}